\documentclass[11pt]{article}

\usepackage[english,activeacute]{babel}

\usepackage[right=3cm,left=3cm,top=3cm,bottom=3cm]{geometry}
\usepackage{float}

\usepackage[utf8]{inputenc}
\usepackage{amsmath,amssymb,amsthm, 
	verbatim, 
	fancyhdr,amsfonts}

\usepackage{needspace}
\usepackage{float}
\usepackage{tikz}

\usepackage{amsmath}
\usepackage{amsthm}
\usepackage{amsfonts,amsmath,amssymb}
\usepackage{enumerate,verbatim,color}
\usepackage{epsfig}
\usepackage{textcomp}
\usepackage{graphicx}
\usepackage{amsfonts}
\usepackage{mathrsfs}
\usepackage{upgreek}
\usepackage{mathtools}

\newcommand{\virgolette}[1]{``#1''}
\newtheorem{theorem}{Theorem}[section]
\newtheorem{coro}[theorem]{Corollary}

\newtheorem{remark}[theorem]{Remark}
\newtheorem{prop}[theorem]{Proposition}

\theoremstyle{definition}

\theoremstyle{remark}

\newcommand{\iii}{{\, \vert\kern-0.25ex\vert\kern-0.25ex\vert\, }}

\DeclareMathOperator{\spn}{span}
\newcommand{\ffi}{\varphi}

\newcommand{\vf}{\varphi}

\newcommand{\ds}{\displaystyle}

\newcommand{\Z}{\mathbb{Z}}
\newcommand{\N}{\mathbb{N}}
\newcommand{\R}{\mathbb{R}}
\newcommand{\Q}{\mathbb{Q}}
\newcommand{\C}{\mathbb{C}}

\newcommand{\Hi}{\mathscr{H}}

\newcommand{\Si}{\mathscr{S}}

\newcommand{\dd}{\partial}

\newcommand{\ra}{\rangle}
\newcommand{\la}{\langle}

\newcommand{\Ti}{\mathcal{T}}

\title{Networks of sign-changing metamaterials: existence and spectral properties}
\author{
{Ka\"{\i}s Ammari,\footnote{LR Analysis and Control of PDEs, LR 22ES03, Department of Mathematics, Faculty of Sciences of Monastir, University of Monastir, Tunisia; e-mail: kais.ammari@fsm.rnu.tn}}
\and{Alessandro Duca,\footnote{Universit\'e de Lorraine, CNRS, INRIA, IECL, F-54000 Nancy, France;  e-mail: alessandro.duca@inria.fr}}
\and{Eric Bonnetier,\footnote{Institut Fourier, Universit\'e Grenoble Alpes, 100 Rue des Math\'ematiques, 38610 Gi\`eres, France;  e-mail: eric.bonnetier@univ-grenoble-alpes.fr}}
}

\begin{document}

\maketitle

\begin{abstract}  
  We study composite assemblages of dielectrics and metamaterials with respectively positive and negative material parameters. In the continuum case, for a scalar equation, such media may exhibit so-called plasmonic resonances for certain values of the (negative) conductivity in the metamaterial.
This work investigates such resonances, and the associated eigenfunctions, in the case of composite conducting networks. 
Unlike the continuous media, we show a surprising specific dependence on the geometry of the network of the resonant values. We also study how the problem is affected by the choice of boundary conditions on the external nodes of the structure. 



\end{abstract}

\setcounter{tocdepth}{1}
\tableofcontents
\setcounter{tocdepth}{1}

\section{Introduction}

Negative index materials possess spectacular properties of concentration and localization 
of light, which are very interesting for applications, in particular in optics and communication. 
The first works on negative index materials may date back to Veselago~\cite{Veselago} in 1969.
They have really become a subject of active research about two decades ago, when the
development of lithographic techniques has made possible the manufacturing of 
metallic objects with dimensions smaller than the wavelengths in the range of visible light.
Indeed, in this range of frequencies, objects of small dimensions made of
metals such as silver, gold or aluminium, may have complex dielectric coefficients with
a negative real part and a small imaginary part, and can exhibit concentration effects
(hot spots) when interacting with an incident electromagnetic wave. Such features
find useful applications in medical imaging or microscopy, such as for instance 
single molecule detectors.

\medskip

The simplest model put forward to represent the associated resonance effect
is the so-called plasmonic electrostatic model~\cite{Grieser} for the voltage 
potential $u$
\begin{equation}\label{eq_plasmonic}
\left\{\begin{array}{clll}
\textrm{div}(a(x) \nabla u(x)) &=& 0, & \textrm{in}\; \Omega
\\[6pt]
u(x) &=& 0, &\textrm{in}\; \partial \Omega
\end{array}\right.
\end{equation}
where $\Omega \subset \R^n$ is a domain that represents a composite medium
made of a subset $D$ of metamaterial with conductivity $k < 0$, embedded in
a background dielectric material of conductivity 1, so that 
\begin{eqnarray*}
\alpha(x) &=& 1 + (k-1)1_{D}(x), \quad x \in \Omega.
\end{eqnarray*}
Resonant states are defined as non-trivial solutions to~\eqref{eq_plasmonic},
which, in view of the Lax Milgram Lemma, may only exist when $k < 0$.
\medskip

There is a large body of work on the mathematical analysis of the well-posedness
of~\eqref{eq_plasmonic} (see for instance~\cite{AmmariMillienZhang,AmmariMillien,AmmariBryn,BonnetierZhang1,BonnetierZhang,CostabelStefan,CostableDaugeNicaise,Grieser,Mayergoyz,Mayergoyz1}), where techniques from potential theory or 
from variational analysis are used. The solution to~\eqref{eq_plasmonic} can indeed be represented
as $u = S_D \varphi$, {\it i.e.} as the single layer potential of a function~$\varphi \in H^{-1/2}(\partial D)$, 
and one is led to study an integral equation of the form
$(\lambda I - K^*)\varphi = 0$, where $\lambda = \ds\frac{k+1}{2(k-1)}$ and $K^*$ is the associated Neumann-Poincar\'e operator
\begin{eqnarray*}
K^* \varphi &=&
\ds\int_{\partial D} \ds\frac{\nu(x) \cdot (x-y)}{|x-y|^2} \varphi(y) \,d \sigma(y).
\end{eqnarray*}
\medskip

When $D$ is smooth, $K^*~: H^{-1/2}(\partial D) \rightarrow H^{-1/2}(\partial D)$ is a compact
operator, and its spectrum is formed by a sequence of eigenvalues of finite multiplicity, that
converge to $0$. The latter value $\lambda = 0$ corresponds to a conductivity 
\begin{align}\label{-1}k = -1
\end{align}
inside
the inclusion $D$, which in this situation is the only value of the conductivity for which
the operator $Au = \textrm{div}\Big(1 + (k-1)1_{D}(x) \nabla u\Big)$ looses its Fredholm 
character~\cite{BonnetChesnelCiarlet1,BonnetChesnelCiarlet2}.

\medskip

One could also investigate the well-posedness of second order scalar equations with
sign-changing coefficient using the notion of 
T-coercivity~\cite{BonnetChesnelCiarlet1,BonnetChesnelCiarlet2,karim,ciarlet,NicaiseVenel},
which relates to variational methods. A bilinear form
$a~: H^1(\Omega) \times H^1(\Omega)~\longrightarrow~\R$ is T-coercive if one can find an isomorphism 
$$T~:~H^1(\Omega)~\longrightarrow~H^1(\Omega)$$ such that
$a(T\cdot,\cdot)$ satisfies the hypotheses of the Lax-Milgram Lemma.
In the framework of~\eqref{eq_plasmonic}, T-coercivity may only provide necessary conditions 
on the values of $k$ for which the associated operator is well-posed. However, contrarily
to the integral equation approach, it may be applied to situations where the coefficients
are not merely piecewise constant. The T-coercivity has been also adopted in homogenization problems with sign-changing materials in \cite{kar3, kar2, kar1}.

\subsection*{Sign-changing scalar problems on networks and our results}
In this work, we investigate the behaviour of networks made of mixtures of dielectric 
and metamaterials. We study the associated electrostatic plasmonic equation for star-shaped
and tadpole networks, which takes the variational form~: find $u \in H$, such that for 
all $v \in H$,
\begin{align}\label{prob_intro}
\sum_{i=1}^N \ds\int_{e_i} a_i(x) u_i^\prime(x) v_i^\prime(x)\, dx
= \sum_{i=1}^N \ds\int_{e_i} f_i(x) v_i(x)\, dx.
\end{align}
Here $(e_i)_{i\leq N}$ denotes the $N$ edge of the network,
$f = (f_1, \dots, f_N)$ is a source term, and the field
$u = (u_1, \dots, u_N)$ is sought in a subspace $H$ of 
$\ds \Pi_{i=1}^N H^1(e_i)$,
and satisfies transmission and boundary conditions.
This problem can be studied much more directly than the case of continuous media,
since the PDE can be reduced to a system of ODE's coupled via the transmission 
conditions between the edges of the network. In this work, we proceed as follows.

\begin{itemize}
    \item We start by considering star-graph networks equipped with Dirichlet boundary conditions on the external vertices. We firstly study two-phase networks composed by $N$ equal edges, where the conductivity of the $D$ dielectric edges is $1$ and the negative ones have conductivity $k<0$. We ensure the well-posedness of the problem \eqref{prob_intro} via the T-coercivity when \begin{align}\label{interdit}k\neq -\frac{D}{N-D}\end{align}  (see Theorem \ref{th_star_N}). Notice that the resonant value depends on the geometry of the structure and  differs from the value $-1$, usually appearing for smooth continuous media \eqref{-1}. Secondly, we extend the result to the case of general star-graph networks equipped with Dirichlet boundary conditions (Corollary \ref{coro_star_Nbis}, Theorem \ref{th_star_N_more} and Corollary \ref{coro_star_N_more}). Also in the general case, the geometry of the graph plays an important role in the admissible conductivities for the well-posedness of \eqref{prob_intro}. 

\item Afterwards, we consider star-shaped networks with Dirichlet and Neumann boundary conditions and composed by homogenous materials. We use a different technique from the T-coercivity: we study the transmission relation associated to the variational problem \eqref{prob_intro}. Here, the loss of the Fredholm character occurs at some specific values, in general different from $-1$, again depending on the geometry of the structure. Also in this case, we start by considering two-phase networks (see Theorem \ref{th_star_Neuman_constant}), and later we extend the result to the general case (Corollary \ref{coro_star_Neuman}). An application of our general result covers the case of a star-network with edges of conductivities $k_j$ and lengths $L_j$. In this framework, we denote $I_d$ the edges equipped with Dirichlet boundary conditions and the well-posedness of \eqref{prob_intro} is ensured when 
\begin{align}\label{gen}\sum_{l\in I_d} \frac{k_l}{L_l} \neq 0.\end{align}
The identity \eqref{gen} shows that the well-posedness of \eqref{prob_intro} is not affected by the presence of edges equipped with Neumann boundary conditions, but only by the "Dirichlet part" of the network. We conclude the first part of the work by applying the previous techniques to study the well-posedness for tadpole networks (Theorem \ref {th_tadpole}, Theorem \ref {th_tadpole_weird} and Corollary \ref{th_tadpole_weirdbis}). Notice that this approach can be applied to any network: the PDE is reduced to a linear system of equation representing the transmission conditions between the edges of the network. Clearly, the more the structure is complex, the more the study of the well-posedness of the system becomes complicated.

\item The second part of the work studies the spectral problem associated with \eqref{prob_intro} for the networks considered before (Section \ref{spec_star},  Section~\ref{spec_star_neumann} and Section \ref{spec_tad}). We seek for $(u,\lambda)$ verifying
\begin{eqnarray} \label{eq_varform1}
\forall\; v \in H, \quad
\sum_{i=1}^N \ds\int_{e_i} a_i(x) u_i^\prime(x) v_i^\prime(x)\, dx
&=& \lambda \sum_{i=1}^N \ds\int_{e_i}   u_i(x) v_i(x) dx.
\end{eqnarray}
We characterize the spectral elements of the operator in different frameworks, which has both positive and
negative eigenvalues. This implies that one cannot expect to construct a solution to 
a parabolic equation for a general initial datum, even when the corresponding stationary
operator is T-coercive. One could, however, construct solutions for a Schr\"odinger-type equation
$$i\partial_t u(t,x) + \partial_x \Big( a(x) \partial_x u(t,x) \Big) \;=\; f(t,x).$$
Subsequently, we construct another `natural' family of eigenfunctions, defined as the $(u,\lambda)$'s
that satisfy
\begin{eqnarray} \label{eq_varform}
\forall\; v \in H, \quad
\sum_{i=1}^N \ds\int_{e_i} a_i(x) u_i^\prime(x) v_i^\prime(x)\, dx
&=& \lambda \sum_{i=1}^N \ds\int_{e_i} a_i(x) u_i(x) v_i(x) dx,
\end{eqnarray}
associated with the bilinear form  $(u,v)_a  =  \ds \sum_{i=1}^N \ds\int_{e_i} a_i(x) u_i(x) v_i(x) \, dx. $
This bilinear form is not a scalar product, as the $a_i$'s take positive 
and negative values. We show however that one can find a complete family of eigenfunctions
of this form, which would allow one to define a solution to a `pseudo-parabolic' equation of the form
$$a(x)\partial_t u(t,x) + \partial_x \Big( a(x) \partial_x u(t,x) \Big) \;=\; f(t,x).$$

    \end{itemize}

\subsection*{Acknowledgments}

The authors acknowledge the financial support from VINCI Program 2020 of the “Université Franco Italienne" (UFI) for the project “Graphes quantiques de métamatériaux". The second author would like to thank the colleagues Karim Ramdani and Renata Bunoiu for the fruitful discussions on the behaviour of `indefinite Laplacians' and on the well-posedness of sign-changing scalar problems.

\section{Star-graph networks with Dirichlet boundaries}\label{star_2materials_costant}

Let us consider a star-shaped network composed by $N$ edges of lengths $\{L_j\}_{1 \leq j\leq N}$ some composed by a material with positive conductivity and some with negative conductivity. We 
represent the network with a 
star-graph $\Si$ composed by $N$ edges $\{e_j\}_{1 \leq j\leq N}$.
We denote by $v$ the internal vertex connecting all the edges of $\Si$ and we 
parameterize each $e_j$ 
with a coordinate going from $0$ to its length $L_j$ in $v$.
\begin{figure}[H]
\centering
\includegraphics[width=\textwidth-100pt,height=50pt]{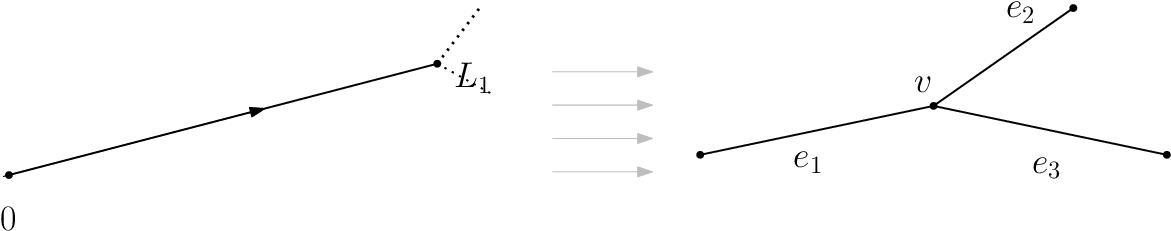}
\caption{The figure shows the parametrization of a star-graph with $3$ edges.}
\label{parametrizzazione}
\end{figure}

We define quantum states on $\Si$ as functions
$\psi=(\psi_1,...,\psi_N)\in 
L^2(\Si):= \ds \prod_{j=1}^NL^2(e_j,\R),$ where each $\psi_j:e_j\rightarrow \R$. The Hilbert space $L^2(\Si)$ 
is equipped with the norm $\|\cdot\|_{L^2}$ induced by the scalar product 
$$(\psi,\ffi)_{L^2}:=\sum_{1 \leq j\leq N}(\psi_j,\ffi_j)_{L^2(e_j,\R)}=\sum_{1 \leq j\leq 
N}\int_{0}^{L_j}{\psi_j}(x)\ffi_j(x)dx,\ \ \ \  \ \ \forall \psi,\ffi\in L^2(\Si).$$
Let $\widehat H^1(\Si):= \ds \prod_{j=1}^N
H^1(e_j,\R)$. We denote by 
$$H^1(\Si)=\{\psi\in 
\widehat H^1(\Si)\ :\ \psi_1(L_1)=...=\psi_N(L_N)\},$$
$$
H^1_0(\Si)=\{\psi\in H^1(\Si)\ :\ \psi_1(0)=....=\psi_N(0)=0\}.
$$
By duality according to the scalar product of $L^2(\Si)$, we define the space $H^{-1}(\Si) $. For every $f=(f_1,...,f_N)\in H^{-1}(\Si)$, we 
investigate the existence of $\psi=(\psi_1,...,\psi_N)\in H^1_0(\Si)$ solution of the following 
stationary problem in 
$L^2(\Si)$
\begin{equation}\label{mainxSNK}\begin{split}
-\dd_x\Big(k_j(x)\dd_x\psi_j(x)\Big)= f_j(x),\ \ \ \ \ \ \ \ &x\in (0,L_j),\ 1 \leq j\leq N,\\
\end{split}
\end{equation}
where $k_j:x\in [0,L_j] \rightarrow \R^*$ are the continuous functions representing the properties of the material composing the network and there exist $I_j,S_j>0$ such that 
$$I_j\leq |k_j(x)|\leq S_j,\ \ \ \ \ \ \ \forall x\in [0,L_j].$$

\subsection{Two-phase networks with constant conductivities}

Let us study the case of star graphs composed of $N$ edges composed of two materials, one with positive conductivity and the other with negative conductivity. In the following theorem, we assume that all the conductivities are constants and that the star graph is equilateral.

\begin{theorem}\label{th_star_N}
Let us consider the problem \eqref{mainxSNK} with $L_1=...=L_N=L\in\R^*_+$. Assume $(k_1,...,k_N)$ be such that  $k_j=1$ for every $1 \leq j\leq D$ and 
$k_j=k^-$ with $k^-<0$ for 
every $D+1\leq j\leq N$. If $$|k^-|\neq \frac{D}{N-D},$$ then for every $f\in H^{-1}(\Si)$, there exists a unique solution $\psi\in H^1_0(\Si)$ solving \eqref{mainxSNK}.
\end{theorem}
\begin{proof}
The proof is based on the T-coercivity approach \cite{BonnetChesnelCiarlet1,BonnetChesnelCiarlet2,karim,ciarlet,NicaiseVenel}. We introduce 
the bilinear form 
$a_k(\cdot,\cdot)$ in $H^1_0(\Si)\times H^1_0(\Si)$ corresponding to \eqref{mainxSNK} and such that, for every $\psi,\ffi\in 
H^1_0(\Si)$,
$$a_k(\psi,\ffi):=\sum_{j=1}^D(\dd_x{\psi_j},\dd_x\ffi_j)_{L^2(0,L)}+k^-\sum_{j=D+1}^N(\dd_x\psi_j,
\dd_x\ffi_j)_{L^2(0,L)}.$$
Denoted $\la\cdot,\cdot\ra$ the duality product between $H^{-1}(\Si)$ and 
$H^1_0(\Si)$, the result is ensured by proving that, for every $f\in H^{-1}(\Si)$, there exists $\ffi\in 
H^1_0(\Si)$ such that 
$$a_k(\ffi,\psi)=\la f,\psi \ra,\ \ \ \ \ \ \ \ \forall \psi\in H^1_0(\Si).$$

\smallskip

\noindent
{\bf 1) Case $|k^-|<D/(N-D)$.} Let us introduce the operator 
\begin{align*}&T_1:(\psi_1,...,\psi_N)\in H^1_0(\Si)\mapsto \\
&\Bigg(\psi_1,...,\psi_D,-\psi_{D+1}+\frac{2}{D}\sum_{j=1}^D\psi_j,...,-\psi_{N}+\frac{2}{D}\sum_{j=1}
^D\psi_j\Bigg)\in H^1_0(\Si).\end{align*}
Thanks to 
the young's inequality, for every $\eta>0$,
\begin{equation*}\begin{split}
a_k(\psi,T_1\psi)&=\sum_{l=1}^D\|\dd_x{\psi_l}\|^2_{L^2(0,L)}-k^-\sum_{l=D+1}^N\|\dd_x{\psi_l}\|^2_{L^2(0,L)}\\
&
+\frac{2k^-}{D} \sum_{l=D+1}^N\sum_{j=1}^D(\dd_x{\psi_l},\dd_x{\psi_j})_{L^2(0,L)}
\\    
&\geq \sum_{l=1}^D\|\dd_x{\psi_l}\|^2_{L^2(0,L)}+|k^-|\sum_{l=D+1}^N\|\dd_x{\psi_l}\|^2_{L^2(0,L)}\\
&
-\frac{|k^-|}{D} 
\sum_{l=D+1}^N\sum_{j=1}^D\Big(\frac{1}{\eta}\|\dd_x{\psi_l}\|^2_{L^2(0,L)}+\eta\|\dd_x{\psi_j}\|^2_{L^2(0,L)}\Big)
\\ 
&\geq 
\Big(1-\frac{|k^-|(N-D)}{D\eta}\Big)\sum_{l=1}^D\|\dd_x{\psi_l}\|^2_{L^2(0,L)}+|k^-|(1-\eta)\sum_{l=D+1}
^N\|\dd_x{\psi_l}\|^2_{L^2(0,L)}.\\  
\end{split}\end{equation*}
Now, for every $|k^-|<D/(N-D)$, we set $\eta=\frac{|k^-|(N-D)}{D(1-\epsilon)}$ with $\epsilon\in 
\big(0,1-\frac{|k^-|(N-D)}{D}\big)$ which 
implies 
$\eta\in (0,1)$ and then 
$$1-\frac{|k^-|(N-D)}{D\eta}=\epsilon>0,\ \ \ \ \  \ \ \ \  \ 1- \eta>0.$$
Finally, $T_1$ is a diffeomorphism from $H^1_0(\Si)$ to itself such that $T_1\circ T_1=Id $ and the bilinear  
form $ \widetilde a_k(\cdot,\cdot):=a_k(\cdot,T_1\cdot)$ is coercive over $H^1_0(\Si) \times H^1_0(\Si)$. Fixed $f\in 
H^{-1}(\Si)$, we call $l(\cdot):=\la f,\cdot \ra$ and $l_1(\cdot):=l(T_1\cdot)$ is a 
continuous linear form in $H^1_0(\Si)$. The Lax-Milgram's theorem yields the existence of a unique $\ffi\in 
H^1_0(\Si)$ such that $\widetilde a_k(\ffi,\psi)=l_1(\psi)$ for every $\psi\in H^1_0(\Si)$ continuously depending on the form 
$l_1$. Since $T_1$ is an isomorphism, there exists a unique $\ffi\in H^1_0(\Si)$ such that 
$a_k(\ffi,\psi)=l(\psi)$ for every $\psi\in H^1_0(\Si)$. This fact concludes 
the proof when $|k^-|<D/(N-D)$.

\smallskip

\noindent
{\bf 2) Case $|k^-|>D/(N-D)$.} Let us introduce the operator 
\begin{align*}&T_2:(\psi_1,...,\psi_N)\in H^1_0(\Si)\mapsto \\
&\Bigg(\psi_1-\frac{2}{N-D}\sum_{j=D+1}^N\psi_j,...,\psi_D-\frac{2}{N-D}\sum_{j=D+1}^N\psi_j,-\psi_{D+1}, 
...,-\psi_N\Bigg)\in H^1_0(\Si).\end{align*}
Thanks to 
the young's inequality, for every $\eta>0$,
\begin{equation*}\begin{split}
a_k(\psi,T_2\psi)&=\sum_{l=1}^D\|\dd_x{\psi_l}\|^2_{L^2(0,L)}+|k^-|\sum_{l=D+1}^N\|\dd_x{\psi_l}\|^2_{L^2(0,L)}\\
&-\frac{2}{N-D}\sum_{j=1}^D \sum_{l=D+1}^N(\dd_x{\psi_j},\dd_x{\psi_l})_{L^2(0,L)}
\\    
&\geq \sum_{l=1}^D\|\dd_x{\psi_l}\|^2_{L^2(0,L)}+|k^-|\sum_{l=D+1}^N\|\dd_x{\psi_l}\|^2_{L^2(0,L)}\\
&-\frac{1}{N-D} 
\sum_{l=D+1}^N\sum_{j=1}^D\Big(\eta\|\dd_x{\psi_l}\|^2_{L^2(0,L)}+\frac{1}{\eta}\|\dd_x{\psi_j}\|^2_{L^2(0,L)}\Big)
\\ 
&\geq 
(1-\eta)\sum_{l=1}^D\|\dd_x{\psi_l}\|^2_{L^2(0,L)}+\Big(|k^-|-\frac{D}{(N-D)\eta}\Big)\sum_{l=D+1}
^N\|\dd_x{\psi_l}\|^2_{L^2(0,L)}.\\  
\end{split}\end{equation*}
Now, for every $|k^-|>D/(N-D)$, we set $\eta=\frac{D}{(N-D)(|k^-|-\epsilon)}$ with $\epsilon\in 
\big(0,|k^-|-\frac{D}{N-D}\big)$ which 
implies 
$\eta\in (0,1)$ and then 
$$|k^-|-\frac{D}{(N-D)\eta}=\epsilon>0,\ \ \ \ \  \ \ \ \  \ 1- \eta>0.$$
The claim is then proved as in the previous case thanks to the coercivity of the bilinear form $a_k(\cdot,T_2\cdot)$.
\end{proof}

We are finally ready to provide the statement in the general case where the star network is composed by edges of $2$ different lengths and materials.

\begin{coro}\label{coro_star_Nbis}
Let us consider the problem \eqref{mainxSNK} with $L_1=...=L_D=L^+\in\R^*_+$ for every $1 \leq j \leq D$ with $D\in \N^*$ such that $D<N$ and $L_{D+1}=...=L_N=L^-\in\R^*_+$. Assume $(k_1,...,k_N)$ be such that $k_j=k^+>0$ for every $1 \leq j \leq D$ and 
$k_j=k^-<0$ 
for 
every $D+1\leq j\leq N$. If $$\frac{|k^-|}{k^+}\neq \frac{D}{L^+}\frac{L^-}{N-D},$$ then for every $f\in H^{-1}(\Si)$, there exists a unique solution $\psi\in H^1_0(\Si)$ solving \eqref{mainxSNK}.
\end{coro}

\begin{proof}
The results follow from Theorem \ref{th_star_N}. First, we use the following change of variable on every edge of the 
graph $$h_j:x\in (0,L_j)\longmapsto y=\frac{x}{L_j}.$$
We substitute the parameters $k^+$ with $\widetilde k^+=\frac{k^+}{L^+}$ and $\widetilde k^-=\frac{k^-}{L^-}$. We also replace $f$ with $\widetilde f=(L_1\widetilde f_1,...,L_N\widetilde f_N)$.
Second, we divide each equation of \eqref{mainxSNK} with respect to $\widetilde k^+$ and we call $\widehat f=\frac{\widetilde f}{\widetilde k^+}$ and $\widehat k^-= \frac{\widetilde k^-}{\widetilde k^+}.$ These two transformations allow us to rewrite the problem \eqref{mainxSNK} in an equivalent one defined on a star graph with edges of equal length. The new problem is still of the form of \eqref{mainxSNK} where the parameters $k_j$ 
are $k_j=1$ for every $1 \leq j\leq D$ and 
$k_j=\widehat k^-$ with $k^-<0$ for 
every $D+1\leq j\leq N$, while $f$ is substitute with $\widehat f$. Finally, the result follows from Theorem \ref{th_star_N}.
\end{proof}

\subsection{General multiphase star-graph networks}

In the following theorem, we ensure well-posedness in the case of equilateral star graphs with non-constant conductivities. The general case where the edges may have different lengths is treated in the subsequent corollary.

\begin{theorem} \label{th_star_N_more}
Let us consider the problem \eqref{mainxSNK} with $L_1=...=L_N=L\in\R$. Assume $(k_1,...,k_N)$ be such that $k_j>0$ for every $1 \leq j \leq D$ with $D\in \N^*$ such that $D<N$ and $k_j<0$ for every $D+1\leq j\leq N$.  If one of the following inequalities is satisfied
$$\Big(\max_{D+1\leq l\leq N} 
S_l\Big)\Bigg(\sum_{l=1}^D\frac{1}{I_l}\Bigg)<\frac{D^2}{N-D},\ \ \ \  \ \ \ \ \ \ \  \ \ \Big(\max_{ l\leq 
D} 
S_l\Big)\Bigg(\sum_{l=D+1}^N\frac{1}{I_l}\Bigg)<\frac{(N-D)^2}{D},$$ then, for every $f\in H^{-1}(\Si)$, there exists a unique solution $\psi\in H^1_0(\Si)$ solving \eqref{mainxSNK}.
\end{theorem}
\begin{proof}
The statement follows by the T-coercivity developed in Theorem \ref{th_star_N}. We consider again
the bilinear form 
$a_k(\cdot,\cdot)$ in $H^1_0(\Si)\times H^1_0(\Si)$ corresponding to \eqref{mainxSNK} and such that, for every $\psi,\ffi\in 
H^1_0(\Si)$,
$$ a_k(\psi,\ffi):=-\sum_{j=1}^N(\dd_x{\psi_j},k_j \dd_x\ffi_j)_{L^2(0,L)}.$$

\smallskip

\noindent
{\bf 1) First case.} Let us consider the operator $T_1:(\psi_1,...,\psi_N)\in H^1_0(\Si) \rightarrow H^1_0(\Si)$ introduced in the proof of Theorem \ref{th_star_N}. Thanks to 
the young's inequality, for every $\eta_{j,l}>0$ with $l\in\{D+1,...,N\}$ and $ j\in \{1,...,D\}$,
\begin{equation*}\begin{split}
a_k(\psi,T_1\psi)&=\sum_{l=1}^N(\dd_x{\psi_l},|k_l|\dd_x{\psi_l})_{L^2(0,L)}-\frac{2}{D} 
\sum_{l=D+1}^N\sum_{j=1}^D(\dd_x{\psi_l},|k_l|\dd_x{\psi_j})_{L^2(0,L)}
\\    
&\geq \sum_{l=1}^N
\big\|\sqrt { |k_l|} \dd_x { \psi_l}\big\|^2_{ L^2(0,L) } \\
&
-\frac{1}{D} 
\sum_{l=D+1}^N\sum_{j=1}^D\Bigg(\eta_{j,l}\Big\|\sqrt{|k_l|}\dd_x{\psi_l}\Big\|^2_{L^2(0,L)}+\frac{1}{\eta_{j,l}}
\Big\|\sqrt { |k_l|} \dd_x { \psi_j}\Big\|^2_{ L^2(0,L) } \Bigg)
\\ 
&\geq 
\sum_{l=1}^D\int_0^L\Big(|k_l(x)|-\frac{1}{D}\sum_{j=D+1}^N\frac{|k_j(x)|}{\eta_{l,j}}\Big)\big|\dd_x{\psi_l}(x)
|^2dx\\
&+  \sum_{l=D+1}^N \Bigg(1-\frac{1}{D}\sum_{j=1}^D\eta_{j,l}\Bigg)
\Big\|\sqrt{|k_l|} \dd_x { \psi_l}\Big\|^2_{ L^2(0,L) }.\\  
\end{split}\end{equation*}
  Let $S^+:=\underset{D+1\leq l\leq N}{\max}\ S_l$. If $K_1:=S^+ \ds \sum_{l=1}^D\frac{1}{I_l}<\frac{D^2}{N-D}$, then 
we 
set $\eta_{j,l}=\frac{S_l(N-D)}{I_j D (1-\epsilon)}$ with $\epsilon\in 
\big(0,1-\frac{K_1(N-D)}{D^2}\big)$ and,
$$|k_l(x)|-\frac{1}{D}\sum_{j=D+1}^N\frac{|k_j(x)|}{\eta_{l,j}}\geq|k_l(x)|- I_l + I_l\epsilon \geq I_l 
\epsilon>0,\ \ \ \ \ \ \ \  \forall x\in (0,L),\ \   \ \ \ \forall l\leq 
D,$$
$$1-\frac{1}{D}\sum_{j=1}^D\eta_{j,l}
\geq 
1-\frac{S^+(N-D)}{D^2}\sum_{j=1}^D\frac{1}{I_j(1-\epsilon)}>
1-\frac{S^+}{K_1}\sum_{j=1}^D\frac{1}{I_j}=0,\ \ \ \ \  \ \ \ \forall D+1\leq l\leq 
N.$$
The last relations yield that there exists $C>0$ such that $a_k(\psi,T_1\psi)\geq C\|\dd_x\psi\|_{L^2}^2$ for every $\psi \in H^1_0(\Si).$ The well-posedness in this case is ensured as in the proof of Theorem \ref{th_star_N}.

\smallskip

\noindent
{\bf 2) Second case.} Let us consider the operator $T_2:(\psi_1,...,\psi_N)\in H^1_0(\Si) \rightarrow H^1_0(\Si)$ introduced in the proof of Theorem \ref{th_star_N}. Thanks to 
the Young's inequality, for every $\eta_{j,l}>0$ with $l\in\{D+1,...,N\}$ and $ j\in \{1,...,D\}$,
\begin{equation*}\begin{split}
a_k(\psi,T_2\psi)&=\sum_{l=1}^N(\dd_x{\psi_l},|k_l|\dd_x{\psi_l})_{L^2(0,L)}-\frac{2}{N-D} 
\sum_{l=1}^D\sum_{j=D+1}^N(\dd_x{\psi_l},|k_l|\dd_x{\psi_j})_{L^2(0,L)}
\\    
&\geq \sum_{l=1}^N
\big\|\sqrt { |k_l|} \dd_x { \psi_l}\big\|^2_{ L^2(0,L) } \\
&-\frac{1}{N-D} 
\sum_{l=1}^D\sum_{j=D+1}^N\Bigg(\eta_{j,l}\Big\|\sqrt{|k_l|}\dd_x{\psi_l}\Big\|^2_{L^2(0,L)}+\frac{1}{\eta_{j,l}}
\Big\|\sqrt { |k_l|} \dd_x { \psi_j}\Big\|^2_{ L^2(0,L) } \Bigg)
\\ 
&\geq  \sum_{l=1}^D \Bigg(1-\frac{1}{N-D}\sum_{j=D+1}^N\eta_{l,j}\Bigg)
\Big\|\sqrt{|k_l|} \dd_x { \psi_l}\Big\|^2_{ L^2(0,L) } + \\
&\sum_{l=D+1}^N\int_0^L\Big(|k_l(x)|-\frac{1}{N-D}\sum_{j=1}^D\frac{|k_j(x)|}{\eta_{j,l}}\Big)\big|\dd_x{\psi_l}(x)
|^2dx.\\  
\end{split}\end{equation*}
Let $S_+:=\underset{ l\leq D}{\max}\ S_l$. If $K_2:=S_+ \ds \sum_{l=D+1}^N\frac{1}{I_l}<\frac{(N-D)^2}{D}$, then we set 
each
$\eta_{j,l}=\frac{S_jD}{I_l (N-D)(1-\epsilon)}$ with $\epsilon\in 
\Big(0,1-\frac{K_2 D}{(N-D)^2}\Big)$ such that 
$$|k_l(x)|-\frac{1}{N-D}\sum_{j=1}^D\frac{|k_j(x)|}{\eta_{j,l}}\geq|k_l(x)|- I_l + I_l\epsilon \geq I_l 
\epsilon>0,\ \ \ \ \ \ \ \  \forall x\in (0,L),\ \   \ \ \ \forall D+1\leq l\leq N,$$
$$1-\frac{1}{N-D}\sum_{j=D+1}^N\eta_{l,j}
\geq 
1-\frac{S_+D}{(N-D)^2}\sum_{j=D+1}^N\frac{1}{I_j(1-\epsilon)}>
1-\frac{S_+}{K_2}\sum_{j=D+1}^N\frac{1}{I_j}=0,\ \ \ \ \  \ \ \ \forall l\leq 
D.$$
The last relations yield that there exists $C>0$ such that $a_k(\psi,T_2\psi)\geq C\|\dd_x\psi\|_{L^2}^2$ for every  $\psi \in H^1_0(\Si).$ Finally, the proof is concluded as the one of Theorem \ref{th_star_N}.
\end{proof}
\begin{coro} \label{coro_star_N_more}
Let us consider the problem \eqref{mainxSNK}. Assume $(k_1,...,k_N)\in \widehat H^1(\Si)$ be such that $k_j>0$ for every $1 \leq j \leq D$ with $D\in \N^*$ such that $D<N$ and $k_j<0$ for every $D+1\leq j\leq N$.  If one of the following inequalities is satisfied
$$\Bigg(\max_{D+1\leq l\leq N} 
\frac{S_l}{L_l}\Bigg)\Bigg(\sum_{l=1}^D\frac{L_l}{I_l}\Bigg)<\frac{D^2}{N-D},\ \ \ \  \ \ \ \ \ \ \  \ \ 
\Bigg(\max_{ l\leq 
D} 
\frac{S_l}{L_l}\Bigg)\Bigg(\sum_{l=D+1}^N\frac{L_l}{I_l}\Bigg)<\frac{(N-D)^2}{D},$$
then for every $f\in H^{-1}(\Si)$, there exists a unique solution $\psi\in H^1_0(\Si)$ solving \eqref{mainxSNK}.
\end{coro}

\begin{proof}
The statement follows from Theorem \ref{th_star_N_more} thanks to the arguments adopted in the proof of Corollary \ref{coro_star_Nbis}.
\end{proof}

\section{Star-networks with Dirichlet and Neumann boundaries}\label{star_neumann}

In this section, we continue the study of the well-posedness of a system as \eqref{mainxSNK} in the case we do not only consider Dirichlet boundary conditions on the external vertices but also Neumann-type conditions. In this case, we use a different approach from the T-coercivity adopted in the previous section and assume constant conductivities on each edge. 

\smallskip

We denote $N_d^+$ and $N_n^+$ the numbers of dielectric edges equipped with Dirichlet and Neumann boundary conditions, respectively. Equivalently, $N_d^-$ and $N_n^-$ are the negative index edges respectively equipped with Dirichlet and Neumann boundary conditions. In details, consider $(k_1,...,k_N)\in\R^N$ be such that 
\begin{equation*}\begin{split}
    \begin{cases}
       k_j>0,\ \ \ \  &\text{for}\ \ \ 1\leq j\leq N_d^++N_n^+,\\
       k_j<0,\ \ \ \  &\text{for}\ \ \ N_d^++N_n^++1\leq j\leq N=N_d^++N_n^++N_d^-+N_n^-.\\
    \end{cases}
\end{split}
\end{equation*}
If we refer to the notation of the previous section $D=N_d^++N_n^+$ and $N-D=N_d^-+N_n^-.$ We also recall that $$\widehat H^1(\Si):= \ds \prod_{j=1}^N
H^1(e_j,\R).$$

\smallskip

Assume $N_d^++N_d^-\geq 1$ (we refer to Remark \ref{homo} for the case $N_d^++N_d^-=0$). For $f=(f_1,...,f_N)\in 
H^{-1}(\Si)$, we study the existence and the unicity of a solution $\psi=(\psi_1,...,\psi_N)\in \widehat H^1(\Si)$ verifying in a weak sense the following system
\begin{equation}\label{mainxSNK_Neumann}\begin{split}\begin{cases}
-k_j \dd_x^2\psi_j(x)= f_j(x), \ \ \ \ \ \ \ \ &\ x\in (0,L_j),\ 1 \leq j\leq N,\\
\psi_j(0)=0,  \ \ \ \ \ \ \ \  \ &\  1 \leq j\leq N_d^+\ \ \ \text{and}\ \ \ N_d^++N_n^++1 \leq j\leq N_d^++N_n^++N_d^-,\\
\dd_x\psi_j(0)=0, \ \ \ \ \ \ \ \  \ &\  N_d^++1 \leq j\leq N_d^++N_n^+\ \ \ \text{and}\ \ \ N_d^++N_n^++N_d^-+1 \leq j\leq N  ,\\
\psi_l(L_l)=\psi_k(L_k), \ \ \ \ \ \ \ \  \ &\  1\leq l,k\leq N ,\\
\sum_{l=1}^{N}k_l\dd_x\psi_l(L_l)=0.  \\
\end{cases}
\end{split}
\end{equation}

\begin{theorem}\label{th_star_Neuman_constant}
Let us consider the problem \eqref{mainxSNK_Neumann} with $L_1=...=L_N=1$. Assume $(k_1,...,k_N)$ be such that  $k_j=k^+>0$ for $1 \leq j\leq N_d^+$ and 
$k_j=k^-<0$ with $k>0$ for $N_d^++N_n^++1\leq j\leq N_d^++N_n^++N_d^-$. If $$\frac{|k^-|}{k^+} \neq  \frac{N_d^+}{N_d^- },$$ then for every $f\in H^{-1}(\Si)$, there exists a unique solution $\psi\in \widehat H^1(\Si)$ solving \eqref{mainxSNK_Neumann}.
\end{theorem}
\begin{proof}

\noindent
{\bf Unicity of the solutions: } We show that, if $\psi$ solves \eqref{mainxSNK_Neumann} with $f=(0,...,0)$, then $\psi=(0,...,0)$. Indeed, by integrating the equation, we obtain
\begin{align*}
 \dd_x\psi_j(x)= \dd_x\psi_j(0), \ \ \ \ \ \ \ \ &\ x\in (0,L_j),\ 1 \leq j\leq N.
\end{align*}
The boundary conditions yields 
\begin{equation}\label{eqq1}\begin{split}
 \dd_x\psi_j(x)= \dd_x\psi_j(0), \ \ \ \ \ \ \ \ &x\in (0,L_j),\ 1 \leq j\leq N_d^+\ \ \ \text{and}\ \ \ N_d^++N_n^++1 \leq j\leq N_d^++N_n^++N_d^-,\\
 \dd_x\psi_j(x)=0, \ \ \ \ \ \ \ \ &x\in (0,L_j),\ N_d^++1 \leq j\leq N_d^++N_n^+\ \ \ \text{and}\ \ \ N_d^++N_n^++N_d^-+1 \leq j\leq N ,
\end{split}\end{equation}
and then
\begin{equation}\label{eqq2}\begin{split}
 \dd_x\psi_j(L_j)= \dd_x\psi_j(0), \ \ \ \ \ \ \ \ &\ 1 \leq j\leq N_d^+\ \ \ \text{and}\ \ \ N_d^++N_n^++1 \leq j\leq N_d^++N_n^++N_d^-,\\
 \dd_x\psi_j(L_j)=0, \ \ \ \ \ \ \ \ &\ N_d^++1 \leq j\leq N_d^++N_n^+\ \ \ \text{and}\ \ \ N_d^++N_n^++N_d^-+1 \leq j\leq N .
\end{split}\end{equation}
The last relations show that on the edges equipped with boundary Neumann boundary conditions, $\psi$ can only be a constant function. We integrate again the equations \eqref{eqq1} and we obtain
\begin{equation*}\begin{split}
 \psi_j(x)=  x\dd_x\psi_j(0)+  \psi_j(0), \ \ \ \ \ \ \ \ &x\in (0,L_j),\ 1 \leq j\leq N_d^+\ \ \ \text{and}\ \ \ N_d^++N_n^++1 \leq j\leq N_d^++N_n^++N_d^-,\\
 \psi_j(x)= \psi_j(0), \ \ \ \ \ \ \ \ &x\in (0,L_j),\ N_d^++1 \leq j\leq N_d^++N_n^+\ \ \ \text{and}\ \ \ N_d^++N_n^++N_d^-+1 \leq j\leq N,
\end{split}
\end{equation*}
and thanks to the validity of the boundary conditions
\begin{equation*}\begin{split}
 \psi_j(x)=  x\dd_x\psi_j(0) , \ \ \ \ \ \ \ \ &x\in (0,L_j),\ 1 \leq j\leq N_d^+\ \ \ \text{and}\ \ \ N_d^++N_n^++1 \leq j\leq N_d^++N_n^++N_d^-,\\
  \psi_j(x)=\psi_j(0), \ \ \ \ \ \ \ \ &x\in (0,L_j),\ N_d^++1 \leq j\leq N_d^++N_n^+\ \ \ \text{and}\ \ \ N_d^++N_n^++N_d^-+1 \leq j\leq N .
\end{split}\end{equation*}
From the last relation and \eqref{eqq2}, we obtain
\begin{equation*} \begin{split}
 \frac{1}{L_j}\psi_j(L_j)= \dd_x\psi_j(0)= \dd_x\psi_j(L_j) , \ \ \ \ \ \ \ \ &\ 1 \leq j\leq N_d^+\ \ \ \text{and}\ \ \ N_d^++N_n^++1 \leq j\leq N_d^++N_n^++N_d^-,\\
\dd_x\psi_j(L_j) =0, \ \ \ \ \ \ \ \ &\ N_d^++1 \leq j\leq N_d^++N_n^+\ \ \ \text{and}\ \ \ N_d^++N_n^++N_d^-+1 \leq j\leq N ,
\end{split}\end{equation*}
We use the last identities in the boundary conditions in the internal vertex appearing in the problem \eqref{mainxSNK_Neumann} and we obtain the following linear system of equation
\begin{equation}\label{boundary}\begin{split}\begin{cases}
\psi_l(L_l)=\psi_k(L_k), \ \ \ \ \ \ \ \  \ &\  1\leq l,k\leq N ,\\
\sum_{l=1}^{N_d^+} \frac{k_l}{L_j} \psi_j(L_j)+\sum_{l=N_d^++N_n^++1}^{N_d^++N_n^++N_d^-}\frac{k_l }{L_j} \psi_j(L_j)=0.  
\end{cases}
\end{split}
\end{equation}
We write \eqref{boundary} in a matrix form as $M\cdot v=0$ where $ v=(\psi_1(L_1),...,\psi_N(L_N))$ and 
\begin{eqnarray}\label{M}
M=\left(\begin{array}{*{12}c}
1 & - 1 & 0 & ... & ...& ... &  ... & ... & ...& ...& ...& ... \\
0 & 1 & - 1 & 0 & ... & ... & ... & ... & ...& ...& ...&  ... \\
... &  ... &  ...  & ... & ...  & ...& ...& ...& ...& ...  & ... & ... \\
... &  ... &  ...  & ... & ...  & ...& ...& ...& ...& 0  & -1 & 1 \\
 \frac{k_1}{L_1} & ... &  \frac{k_{N_d^+}}{L_{N_d^+}} &  0 & ... & 0 & \frac{k_{N_d^++N_n^++1}}{L_{N_d^++N_n^++1}} & ... & \frac{k_{N_d^++N_n^++1}}{L_{N_d^++N_n^++1}}&  0 & ... & 0 \end{array}\right).  \end{eqnarray}
Notice that $det(M)=\sum_{l=1}^{N_d^+} \frac{k_l}{L_j} +\sum_{l=N_d^++N_n^+}^{N_d^++N_n^++N_d^-}\frac{k_l }{L_j}$ and $M$ is invertible when
$$\sum_{l=1}^{N_d^+} \frac{k_l}{L_j}  +\sum_{l=N_d^++N_n^+}^{N_d^++N_n^++N_d^-}\frac{k_l }{L_j}\neq 0.$$
We recall the assumptions on the parameters $k_j$ and the lengths $L_j$ so that the last identity becomes
$$N_d^+ k^+ + N_d^- k^-\neq 0.$$
Finally, we have $\psi=(0,...,0)$ when $\frac{-k^-}{k^+} \neq  \frac{N_d^+}{N_d^- }$, which implies the unicity of solutions.

\smallskip

\noindent
{\bf Existence of the solutions: } We solve the problem \eqref{mainxSNK_Neumann} by integrating the equation twice such as in the previous step. Here, we consider a general source term $f$ and the boundary conditions verified in the internal vertex of $\Si$ correspond to a linear system that can be written in a matrix form as follows
\begin{align}\label{eqqqq}M v + F = 0.\end{align}
The matrix $M$ is the same introduced in \eqref{M}, while $F=(0,...,0, F_N)$ is defined by
\begin{align*}
F_N=&\sum_{j=1}^{N_d^+}\int_0^1  \Big(f_j(t) -\int_0^t  f_j(s)ds \Big)dt+\sum_{j=N_d^++1}^{N_n^+}\int_0^1  f_j(t)dt \\
&+\sum_{j=N_d^++N_n^++1}^{N_d^++N_n^++N_d^-}\int_0^1  \Big(f_j(t) -\int_0^t  f_j(s)ds \Big)dt+\sum_{j=N_d^++N_n^++N_d^-+1}^{N}\int_0^1  f_j(t)dt .
\end{align*}
The equation \eqref{eqqqq} is solvable when $M$ is invertible which is guaranteed when $\frac{-k^-}{k^+} \neq  \frac{N_d^+}{N_d^- }.$

\end{proof}

\begin{remark}\label{homo}
    Notice that when only Dirichlet boundary conditions are verified, we have $N_n^+=N_n^-=0$ and the last theorem provides the same results of Corollary \ref{coro_star_Nbis} since $N_d^+=D$ and $N_d^-=N-D$.
    On the other hand, when the graph is equipped with only Neumann boundary conditions and $N_d^++N_d^-=0$, the well-posedness is guaranteed independently of the choice of the parameters $k_j$. However, the result is only valid for source terms $f$ such that $\int_{\Si} f(x)dx=0$.  

\end{remark}

\begin{remark}\label{2edges_neuman}
A simple but still interesting consequence of Theorem \ref{th_star_Neuman_constant} is the well-posedness in the case of an interval. Let $a \in (0,1)$ and $k_1,k_2\in\R^*$. We consider the problem in $H^1(0,a)\times H^1(a,1)$
\begin{equation*}\begin{split}\begin{cases}
-k_1 \dd_x^2\psi(x)= f (x), \ \ \ \ \ \ \ \ &\ x\in (0,a),\\
-k_2 \dd_x^2\psi(x)= f(x), \ \ \ \ \ \ \ \ &\ x\in (a,1),\\
\end{cases}
\end{split}
\end{equation*}
 with $f\in  H^{-1}(0,a)\times H^{-1}(a,1)$ equipped with the boundary condition at the internal point
\begin{equation*}\begin{split} 
\psi(a^-)=\psi(a^+), \ \ \ \ \ \  \ k_1\dd_x\psi_l(a^-)=k_2\dd_x\psi (a^+),  \\
\end{split}\end{equation*}
and at the external points
\begin{equation*} 
\psi(0)=\dd_x\psi(1)=0\  \ \ \ \ \ \ \ \  \text{or}\ \ \ \ \ \ \  \ \  \dd_x\psi(0)=\psi(1)=0.\\
\end{equation*}
The existence and the unicity of the solution $\psi$ of the problem is guaranteed independently of the constants $k_1$ and $k_2$ (also with different signs). Notice that the well-posedness is ensured only for $k_1/k_2 \neq -1$ when we study the same problem with only Dirichlet boundary conditions on the external points. Here, the presence of one Neumann boundary condition yields the result for every $k$ also negative. The same property is also true in the case only Neumann boundary conditions are verified when we choose a source term $f$ such that $\int_0^1f(x)dx=0.$
\end{remark}

\begin{coro} \label{coro_star_Neuman}
Let us consider the problem \eqref{mainxSNK_Neumann}. Assume $(k_1,...,k_N)$ be such that  
$$\sum_{l=1}^{N_d^+} \frac{k_l}{L_j}  +\sum_{l=N_d^++N_n^+}^{N_d^++N_n^++N_d^-}\frac{k_l }{L_j}\neq 0.$$
 For every $f\in H^{-1}(\Si)$, there exists a unique solution $\psi\in \widehat H^1(\Si)$ solving \eqref{mainxSNK_Neumann}.
\end{coro} 
\begin{proof}
The result is a direct consequence of the proof of Theorem \ref{th_star_Neuman_constant}.
\end{proof}

\section{Tadpole-shaped networks}\label{section_general_graphs}

\subsection{A simple case: two-phase network}\label{simple}
Consider a tadpole graph $\Ti$ composed of two edges 
$e_1$ and $e_2$ of lengths $L_1$ 
and $L_2$, respectively. The self-closing edge $e_1$, the \virgolette{head} of the tadpole, is connected to 
another edge $e_2$, the tail, 
in the vertex $v$. The graph $\Ti$ has a symmetry axis denoted by $r$ which goes through $e_2$ and crosses 
$e_1$ in $v$ and the center of $e_1$. The edge $e_1$ is parameterized in the clockwise direction with a 
coordinate going from $0$ to $L_1$ (the 
length of $e_1$). The edge $e_2$ is a half-line equipped with a coordinate starting from $0$ and 
arriving at $L_2$ in $v$\ (see Figure \ref{tadpole} for 
further details).
\begin{figure}[H]
	\centering
	\includegraphics[width=\textwidth-100pt, height=70pt]{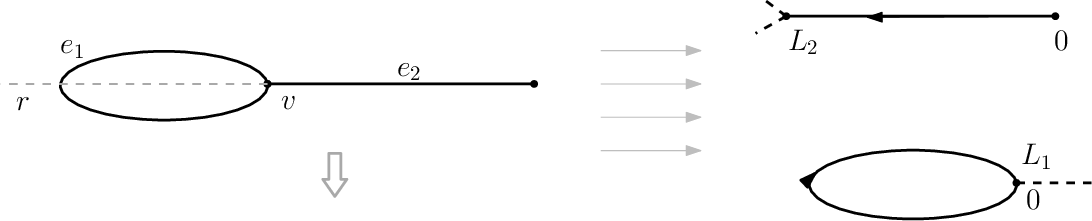}
	\caption{The parametrization of the tadpole graph and its symmetry axis $r$.}
    \label{tadpole}
\end{figure}

\noindent
We consider $\Ti$ as domain of functions $f:=(f_1,f_2):\Ti\rightarrow \R$, such that $f_j:e_j\rightarrow \R$ with 
$j=1,2$. Let $ L^2(\Ti)=L^2((0,L_1),\R)\times L^2((0,L_2),\R)$ be the Hilbert space equipped with the norm $\|\cdot\|$ induced by the scalar 
product
$$(\psi,\ffi)_{L^2}:=\int_{e_1}{\psi_1}(x)\ffi_1(x)dx+\int_{e_2}{\psi_2}
(x)\ffi_2(x)dx,\ \ \ \  \ \ \forall \psi,\ffi\in L^2(\Ti).$$

Let $\widehat H^1(\Ti):= H^1(e_1,\R)\times H^1(e_2,\R) $. By duality according to $L^2(\Ti,\R)$ scalar product, we define the space $H^{-1}(\Ti)$. We consider $k_1,k_2\in\R^*$ and investigate the well-posedness of the following 
stationary problem
\begin{equation}\label{mainxT}\begin{split}
-k_j\dd_x^2\psi_j(x)= f_j(x),\ \ \ \ \ \ \ \ &x\in (0,L_j),\ j\in\{1,2\},\\
\end{split}
\end{equation}
when the following boundary conditions are verified 
\begin{align*}
&\psi_2(0)=0,\ \ \ \  \ \ \ \ \ \ \ \ \psi_1(0)=\psi_1(L_1)=\psi_2(L_2),\\
&k_1 \dd_x\psi_1(L_1)-k_1 \dd_x\psi_1(0)+k_1 \dd_x\psi_1(L_1)=0.\end{align*}

\begin{theorem} \label{th_tadpole}
Let us consider the problem \eqref{mainxT} with $f\in H^{-1}(\Ti,\R)$. Assume $(k_1,k_2)$ be such that $k_1$ and $k_2$ have with different signs. There exists a unique solution $\psi\in \widehat H^1(\Ti)$ solving \eqref{mainxT}.
\end{theorem}
\begin{proof}
Let us assume $L_1=2$ and $L_2=1$. However, the same proof is also valid in the general case. We denote by $R$ the reflection w.r.t. the symmetry axis $r$ as the operator $R$  acting on $L^2(\Ti)$ such that $Rf=(f_1(1-x),f_2)$ for every $f=(f_1, f_2)\in L^2(\Ti)$. The problem \eqref{mainxT} is equivalent to study 
in the same space
\begin{equation}\label{decomposition}\begin{split}\begin{cases}
-k_1 \dd_x^2\psi_1^a(x)-k_1(x)\dd_x^2\psi_1^s(x)= f_1^a(x)+f_1^s(x),\ \ \ \ \ \ \ \ &x\in (0,2),\\
-k_2 \dd_x^2\psi_2^a(x)-k_2(x)\dd_x^2\psi_2^s(x)= f_2^a(x)+f_2^s(x),\ \ \ \ \ \ \ \ &x\in (0,1),\\
\end{cases}\end{split}
\end{equation}
with $$\psi^a=(\psi^a_1,\psi^a_2)=\frac{\psi-R\psi}{2},\ \ \ 
\psi^s=(\psi^s_1,\psi^s_2)=\frac{\psi+R\psi}{2},$$
$$f^a=(f^a_1,f^a_2)=\frac{f-Rf}{2},\ \ \ \ 
f^s=(f^s_1,f^s_2)=\frac{f+Rf}{2}.$$ We notice that $\psi^a$ and $f^a$ are antisymmetric w.r.t. the 
axis $r$ and 
they vanish on $e_2$, while $\psi^s$ and $f^s$ are symmetric. This new formulation yields that 
studying the 
problem \eqref{mainxT} can be done by considering $\psi$ and $f$ antisymmetric at first, and after 
symmetric.

\smallskip

 First, the problem \eqref{mainxT} with $\psi\in \widehat H^1(\Ti)$ antisymmetric is equivalent to the following problem $f\in H^{-1}((0,1),\R)$
\begin{equation*}\begin{split}\begin{cases}
-k_1\dd_x^2\psi_1(x)= f_1(x),\ \ \ \ \ \ \ \ \  \ & x\in (0,1),\\
\psi(0)=\psi(1)=0.
\end{cases}
\end{split}
\end{equation*}
Indeed, $\psi^a$ and $f^a$ are equal to zero on $e_2$ thanks to the definition of $R$. The last problem is clearly well-posed independently of the choice of $k_1$ (and obviously of $k_2$).

\smallskip

Second, we consider the trivial star-graph $\Si$ composed of two connected edges of length $1$.  The problem
\eqref{mainxSNK} with $\psi\in \widehat H^1(\Ti)$ symmetric w.r.t. the axis $r$ is equivalent to studying the following one in $
 H^1(0,1)\times H^1(0,1)$
\begin{equation*}\begin{split}\begin{cases}
-k_1\dd_x^2\psi_1(x)= f_1(x), & x\in (0,1),\\
-k_2\dd_x^2\psi_2(x)= f_2(x), & x\in (0,1),\\
\end{cases}\end{split}
\end{equation*}
equipped with the boundary conditions 
$$\dd_x\psi_1(0)=\psi_2(0)=0,\ \ \ \ \ \  \ \ \ \psi_1(L_1)=\psi(L_2),\ \ \ \ \ \  \ \ \ 2 k_1 \dd_x\psi_1(L_1)+k_2 \dd_x\psi_2(L_2)=0.$$
However, if we denote $\tilde k_1 = 2 k_1$ and  $\tilde f_1= 2 f_1$, then we obtain \eqref{mainxSNK_Neumann}. We refer to Theorem \ref{th_star_Neuman_constant} and,  as explained in Remark \ref{2edges_neuman}, the existence and the unicity of solutions hold independently of the choice of $k_1$ and $k_2$. 
\end{proof}

\begin{remark}
The well-posedness of \eqref{mainxT} when the Neumann boundary condition is verified on the external vertex is also ensured independently of the choice of the parameters $k_j$ when we consider the source term $f$ such that $\int_{\Ti}f(x)dx=0$ thanks to Remark \ref{homo}.  
\end{remark}

\subsection{A more complex case: tree-phase tadpole networks}
Let us parametrize a tadpole graph $\Ti$ differently from the previous chapter (see Figure \ref{tadpolebis}). We use three edges 
$\widetilde e_1$, $\widetilde e_2$ and $\widetilde e_3$ so that $|\widetilde e_1|=\widetilde L_1$, $|\widetilde e_2|=\widetilde L_2$ and $|\widetilde e_3|=\widetilde L_3$, and we denote it $\widetilde \Ti$. We consider the \virgolette{head} of the tadpole ($e_1$ in the previous subsection) as composed of two edges $\widetilde e_1$ and $\widetilde e_2$. The two edges are connected on the left in a vertex $\widetilde v$ parametrized with the coordinate $0$ while they are both connected on the right with the edge $\widetilde e_3$ at the vertex $v$. The coordinate $0$ of the edge $\widetilde e_3$ is located at the external point.

\begin{figure}[H]
	\centering
	\includegraphics[width=\textwidth-100pt, height=70pt]{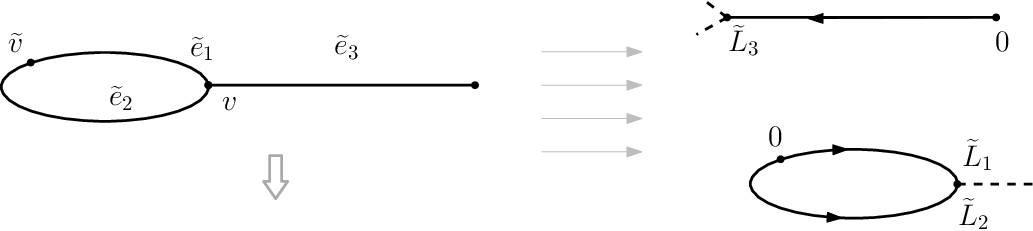}
	\caption{The new parametrization of the tadpole graph.}
    \label{tadpolebis}
\end{figure}

We consider $\widetilde\Ti$ as domain of functions $f:=(f_1,f_2,f_3):\widetilde\Ti\rightarrow \R$, such that $f_j:\widetilde e_j\rightarrow \R$ with 
$j=1,2,3$. Let $ L^2(\widetilde\Ti)=L^2(\widetilde e_1,\R)\times L^2(\widetilde e_2,\R)\times L^2(\widetilde e_3,\R)$ be the Hilbert space equipped with the corresponding norm and scalar 
product. Let $\widehat H^1(\widetilde\Ti):= H^1(\widetilde e_1,\R)\times H^1(\widetilde e_2,\R)\times H^1(\widetilde e_3,\R) $. By duality, we define the space $H^{-1}(\widetilde\Ti)$. We investigate the well-posedness of

\begin{equation}\label{mainxT_weird}\begin{split}\begin{cases}
-\widetilde k_1\dd_x^2\psi_1(x)= f_1(x),\ \ \ \ \ \ \ \ &x\in (0,\widetilde L_1),\\
-\widetilde k_2\dd_x^2\psi_2(x)= f_2(x),\ \ \ \ \ \ \ \ &x\in (0,\widetilde L_2),\\
-\widetilde k_3\dd_x^2\psi_3(x)= f_3(x),\ \ \ \ \ \ \ \ &x\in (0,\widetilde L_3).\\
\end{cases}
\end{split}
\end{equation}
equipped with the boundary conditions 
\begin{align}
&\psi_3(0)=0, \ \ \ \nonumber\\ 
&\psi_1(0)=\psi_2(0),\label{1}\\
&\widetilde k_1\dd_x\psi_1(0)+\widetilde k_2\dd_x\psi_2(0)=0,\label{2}\\
&\psi_1(\widetilde L_1)=\psi_2(\widetilde L_2)=\psi_3(\widetilde L_3),\label{3}\\
&\widetilde k_1 \dd_x\psi_1(\widetilde L_1)+\widetilde k_2 \dd_x\psi_2(\widetilde L_2)+\widetilde k_3 \dd_x\psi_3(\widetilde L_3)=0.\label{4}\end{align}

\begin{theorem} \label{th_tadpole_weird}
Let us consider the problem \eqref{mainxT_weird} with $\widetilde L_1=\widetilde L_2=\widetilde L_3=1$. Assume $(\widetilde k_1,\widetilde k_2,\widetilde k_3)=(k^-,k^+,k^+)$ be such that $k^+>0$ and 
$k^-<0$. If $$\frac{|k^-|}{k^+} \neq 1,$$
then, for every  $f\in H^{-1}(\widetilde\Ti,\R)$, there exists a unique solution $\psi\in \widehat H^1(\widetilde\Ti)$ solving \eqref{mainxT_weird}.
\end{theorem}
\begin{proof}
The proof follows from the same arguments adopted in the proof of Theorem \ref{th_star_Neuman_constant}. We integrate each component of the equation \eqref{mainxT_weird} and we impose the boundary conditions. As in the mentioned proof, we study the linear system obtained by the boundary conditions on the two internal vertices $v$ and $\tilde v$. In this case, we obtain a vector $v$ such that
$$ v=(\psi_1(1),\psi_2(1),\psi_3(1),\psi_1(0),\psi_2(0))$$
and a matrix $M$.
\begin{eqnarray*}
M=\left(\begin{array}{*{12}c}
1 & - 1 & 0 & 0 & 0  \\
0 & 1 & - 1 & 0 & 0  \\
0  & 0 & 0 & 1 & - 1  \\
\frac{\widetilde k_1}{\widetilde L_1} &  \frac{\widetilde k_2}{\widetilde L_2} &  0  & -\frac{\widetilde k_1}{\widetilde L_1} & -\frac{\widetilde k_2}{\widetilde L_2}   \\
\frac{\widetilde k_1}{\widetilde L_1} &  \frac{\widetilde k_2}{\widetilde L_2} &  \frac{\widetilde k_3}{\widetilde L_3}  & -\frac{\widetilde k_1}{\widetilde L_1} & -\frac{\widetilde k_2}{\widetilde L_2}   \\
\end{array}\right).  \end{eqnarray*}
Notice that the first two lines follow from the boundary conditions \eqref{3}, the third from \eqref{1}, the fourth from \eqref{2}, and the fifth from \eqref{4}. The invertibility of the matrix $M$ is guaranteed when 
\begin{align*} det(M)=\frac{\widetilde k_1}{\widetilde L_1}\frac{\widetilde k_3}{\widetilde L_3}+\frac{\widetilde k_2}{\widetilde L_2}\frac{\widetilde k_3}{\widetilde L_3}\neq 0.\end{align*}
We use the hypotheses on the conductivities $\widetilde k_j$ and the lengths $\widetilde L_j$. The existence and unicity are again implied by the invertibility of $M$ which is guaranteed when $det(M)= -k^+(k^++k^-)\neq 0$.
\end{proof}

\begin{coro} \label{th_tadpole_weirdbis}
Let us consider the problem \eqref{mainxT_weird} If $$ -\frac{\widetilde k_1}{\widetilde k_2}\neq \frac{\widetilde L_1}{\widetilde L_2},$$
then, for every  $f\in H^{-1}(\widetilde\Ti,\R)$, there exists a unique solution $\psi\in \widehat H^1(\widetilde\Ti)$ solving \eqref{mainxT_weird}.
\end{coro}
\begin{proof}
The result is a direct consequence of the proof Theorem \ref{th_tadpole_weird}.
\end{proof}

\begin{remark}
If we consider \eqref{mainxT_weird} and we change the boundary condition at the external vertex with a Neumann one, then the existence and unicity of solutions is ensured independently of the choice of the conductivities when $f$ is such that $\int_{\widetilde \Ti}f(x)dx=0$.  
\end{remark}

\begin{remark}
The techniques adopted in the proof of  Theorem \ref{th_star_Neuman_constant} and Theorem \ref{th_tadpole_weird} can be adopted to study the well-posedness of a problem as \eqref{prob_intro} defined on any network. It is sufficient to integrate the associated equations and rewrite the boundary conditions on the internal vertices as a linear system of the form $M\cdot v=F$. The matrix $M$ keeps track of the structure of the network, and its invertibility infers the well-posedness of the problem \eqref{prob_intro}. Clearly, this task can be really challenging when we study networks composed by several edges and with a complex structure.\end{remark}

\section{Spectral representation for star-networks with Dirichlet boundaries}\label{spec_star}
This section aims to provide some spectral properties to the problem \eqref{mainxSNK} defined on a star-graph $\Si$ when $k_j=1$ for $j\leq D$, while $k_j=k^-<0$ for $D+1\leq k\leq   N$. Let us consider the following quadratic forms
\begin{eqnarray*}
(\psi,\vf)_k &=& \sum_{j=1}^D \ds\int_{e_j} \psi_j \vf_j 
\;+\;\sum_{j=D+1}^N \ds\int_{e_j}  k^-\psi_j \vf_j.
\end{eqnarray*}
Notice that the quadratic form $(\cdot,\cdot)_k$ is also a scalar product only when $k^- > 0$.

\begin{itemize}
\item A natural problem concerning \eqref{mainxSNK} is the following eigenvalue problem: find $\lambda$ and $\psi \in H^1_0({\Si})$
such that, for all $\vf \in  H^1_0({\Si})$,
\begin{eqnarray}\label{spec1}
a_k(\psi,\vf) &:=& \sum_{j=1}^D \ds\int_{e_j}  \psi_j^\prime \vf_j^\prime
\;+\;
 \sum_{j=D+1}^N \ds\int_{e_j}  k^-\psi_j^\prime \vf_j^\prime
\;=\;  \lambda  (\psi,\vf).
\end{eqnarray}
It corresponds to the study of the eigenfunctions of the following operator in $L^2(\Si)$:
$$A\psi =\big(-\partial_x^2\psi_1\;,\;...\;,\;-\partial_x^2\psi_D\;,\;-k^-\partial_x^2\psi_{D+1}\;,\;...\;,\;-k^-\partial_x^2\psi_N\;\big) \ \ \  \ \ \text{for}\ \ \  \psi\in D(A) \ \ \ \ \ \ \text{with}$$
\begin{align*}
D(A)=\Bigg\{\psi\in \prod_{j=1}^N H^{2}(e_j,\R)\ :\ &\psi_1(0)=...=\psi_N(0)=0,\ \ \psi_1(|e_1|)=...=\psi_N(|e_N|),\\
& \sum_{j=1}^D \partial_x\psi_j(|e_j|) + k^-\sum_{j=D+1}^N \partial_x\psi_j(|e_j|) =0\Bigg\}.
\end{align*}
\begin{prop}\label{self_adjoint_A}
The operator $A$ is self-adjoint and admits compact resolvent when $$|k^-|\neq \frac{D}{N-D}.$$
\end{prop}
\begin{proof}
The proof of the self-adjointness of the operator $A$ is very standard. Indeed, the integration by parts in the one-dimensional integrals defining $\la\cdot,\cdot\ra_{L^2}$ leads to the symmetry of the operator at first. Second, it is straightforward to show that $D(A)=D(A^*)$. Finally, the fact that $A$ admits compact resolvent is a direct consequence of Theorem \ref{th_star_N}.
    
\end{proof}
\item On the other hand, it can also be interesting to study the following ``alternative''  eigenvalue problem: find $\lambda$ and $\psi \in H^1_0({\Si})$
such that, for all $\vf \in  H^1_0({\Si})$,
\begin{eqnarray}\label{spec2}
a_k(\psi,\vf)
\;=\;  \lambda  (\psi,\vf)_k.
\end{eqnarray}
This problem can be studied by seeking for an eigenfunction $\psi$ satisfying $\widetilde A \psi=\lambda  \psi,$ where the operator $\widetilde A$ is
$$\widetilde A\psi =\big(-\partial_x^2\psi_1\;,\;...\;,\;-\partial_x^2\psi_N\;\big) \ \ \  \ \ \text{for}\ \ \  \psi\in D(\widetilde A) \ \ \ \ \ \ \text{with}$$
\begin{align*}
D(\widetilde A)=\Bigg\{\psi\in \prod_{j=1}^N H^{2}(e_j,\R)\ :\ &\psi_1(0)=...=\psi_N(0)=0,\ \ \psi_1(|e_1|)=...=\psi_N(|e_N|),\\
& \sum_{j=1}^D \partial_x\psi_j(|e_j|) + k^-\sum_{j=D+1}^N \partial_x\psi_j(|e_j|) =0\Bigg\}.
\end{align*}

\begin{remark}\label{non_self_adjoint_A}
Notice that $\widetilde A$ is neither self-adjoint nor symmetric, on the contrary of $A$  (Proposition \ref{self_adjoint_A}). Indeed, there exist $\psi,\varphi\in D(\widetilde A)$, such that $\la \widetilde A\psi,\varphi\ra_{L^2}\neq \la \psi, \widetilde A\varphi\ra_{L^2}$. 
\end{remark}

\end{itemize}

\subsection{Standard spectral representation for equilateral star-networks}\label{spec_diff}
Let us focus on the spectral problem \eqref{spec1} when the length of all the edges is equal to $L$. To make things simple, we assume $L = 1$. The other cases can be treated in the same manner. 
To solve the spectral problem, we set $\xi =   i s \mu$ with $s=\sqrt{|k^-|}$. The eigen-elements $\psi$ and $\xi$ 
equivalently satisfy
\begin{eqnarray*}
\psi  &=& \big(\alpha_1 \sin(\xi x),...,\alpha_{D}\sin(\xi x),
\beta_{D+1} \sin(\mu x), ...,\beta_{N} \sin(\mu x)\big),
\end{eqnarray*}
where the coefficients $\alpha=(\alpha_1,...,\alpha_D)$ and $ \beta=(\beta_{D+1},...,\beta_N)$ satisfy
\begin{equation}\label{condis}
\begin{split}
\begin{cases}\alpha_j \sin(\xi) = \alpha_{j+1}\sin(\xi),\ \ \ \  \ \ \ &1\leq j \leq D-1, \\
\alpha_D \sin(\xi) = \beta_{D+1}\sin(\mu), \\
\beta_{j} \sin(\mu) = \beta_{j+1}\sin(\mu),\ \ \ \  \ \ \ &D+ 1\leq j \leq N-1, \\
\ds \sum_{j=1}^D \alpha_j \xi\cos(\xi) +k^- \ds \sum_{j=D+1}^N\beta_j \mu \cos(\mu)=0.\\
\end{cases}
\end{split}
\end{equation}

\begin{itemize}

\item Notice that when $\sin(\xi)=0$, we find the eigenvalues $\lambda$ of the form $n^2\pi^2$ with $n\in\N^*$. In this case, $\beta_j=0$ with $D+1\leq j\leq N$, thanks to the continuity condition in \eqref{condis}, and we can define the corresponding eigenfunctions of the form
$$\big(\alpha_1 \sin(n \pi x),...,\alpha_{D}\sin(n\pi x),
0,...,0\big),$$
where $\alpha_j$ are some normalizing constants such that 
$$\sum_{j=1}^D \alpha_j =0.$$
We observe that each eigenspace associated to such eigenvalues has dimension $D-1$ and then the eigenvalues are multiple if $D\geq 3$. We denote by $(\nu_{j})_{j\in\N^*}$ such eigenvalues counted with their multiplicity, and we denote by $(\varphi^{j})_{j\in\N^*}$ some corresponding normalized eigenfunctions. 

\item Similarly to the previous case, when $\sin(\mu)=0$, we find the eigenvalues $\lambda$ of the form $ k^- {n^2\pi^2}$ with $n\in\N^*$. Here, $\alpha_j=0$ with $1\leq j\leq D$ and the corresponding eigenfunctions have the form
$$\big(0,...,0,\beta_{D+1} \sin(n \pi x),...,\beta_{N}\sin(n\pi x)\big),$$
where $\beta_j$ are normalizing constants such that 
$$\sum_{j=D+1}^N \beta_j =0.$$
Each eigenspace has dimension $N-D-1$ and the eigenvalues are multiple if $N-D\geq 3$. We denote by $(\theta_{j})_{j\in\N^*}$ such eigenvalues counted with their multiplicity, and we denote by $(\rho^{j})_{j\in\N^*}$ some corresponding normalized eigenfunctions. 

\item When $\sin(\xi),\sin(\mu)\neq 0$, the conditions \eqref{condis} can be rewritten as
\begin{eqnarray*}
\left(\begin{array}{*{8}c}
1 & - 1 & 0 & ... & ...& ... &  ... & ... \\
0 & 1 & - 1 & 0 & ... & ... & ... & ... \\
... &  ... &  ...  & ... & ...  & ...  & ... & ... \\
...& ... & 0 & \sin(\xi)  & -\sin(\mu)& 0 & ... & ... \\
...& ...  &  ... &  ...  &  ... &  ... &  ... & ... \\
... & ... & ... & ... & ... & 0  & 1 & -1 \\
 \xi\cos(\xi) & ... &  \xi\cos(\xi) &  \xi\cos(\xi) & k^-\mu \cos(\mu) & k^-\mu \cos(\mu) & ... & k^-\mu \cos(\mu) \end{array}\right) \left(\begin{array}{c}  \alpha \\  \beta \end{array}\right)
&=& 0 \end{eqnarray*}
and its determinant yields the dispersion relation 
\begin{eqnarray*}
D  \xi \cos(\xi) \sin(\mu) + (N-D) k^- \mu \sin(\xi) \cos(\mu) &=& 0.
\end{eqnarray*}
We recall that $\xi =   i s \mu$ and then 
\begin{eqnarray*}
D s \cosh(s\mu ) \sin( \mu) + (N-D) k^- \sinh(s\mu ) \cos( \mu) &=& 0.
\end{eqnarray*}
Now, we notice that $s=\sqrt{|k^-|}$ which implies $s = -\frac{k^-}{s}$ and, with $x= \frac{D}{s(N-D)}$,
\begin{eqnarray}\label{disp_rel}
 \cos(\mu)\sinh(s\mu) - x \sin(\mu) \cosh(s\mu)= 0.
\end{eqnarray}
Assuming that $\mu = a+ib$, with $a,b \in \R$, this expression can be rewritten in
the form
\begin{eqnarray*}
0 &=&
\left[\cos(a)\cosh(b) - i\sin(a)\sinh(b)\right]
\left[\sinh(sa)\cos(sb) + i\cosh(sa)\sin(sb)\right]\\
&&-
x\left[\sin(a)\cosh(b) + i\cos(a)\sinh(b)\right]
\left[\cosh(sa)\cos(sb)+i \sinh(sa)\sin(sb)\right].
\end{eqnarray*}
Taking the real and imaginary part of this relation yields two
equations of the form
\begin{eqnarray*}
A \left(\begin{array}{c} x \\ 1 \end{array}\right)
&=& \left(\begin{array}{c} 0 \\ 0 \end{array}\right),
\end{eqnarray*}
where $A$ is a $2 \times 2$ matrix composed by the elements 
\begin{eqnarray*}
A_{1,1}=\cos(a) \sinh(b) \sinh(a s) \sin(b s) -  \sin(a) \cosh(b) \cosh(a s) \cos(b s),\\
A_{1,2}=\cos(a) \cosh(b) \sinh(a s) \cos(b s) + \sin(a) \sinh(b) \cosh(a s) \sin(b s),\\
A_{2,1}=- \cos(a) \sinh(b) \cosh(a s) \cos(b s) - \sin(a) \cosh(b) \sinh(a s) \sin(b s),\\
A_{2,2}= \cos(a) \cosh(b) \cosh(a s) \sin(b s) - \sin(a) \sinh(b) \sinh(a s) \cos(b s)
\end{eqnarray*}
and its determinant simplifies to
\begin{eqnarray*}
\det(A) &=& 
\cosh(sa)\sinh(sa)\cosh(b)\sinh(b)
-\sin(a)\cos(a) \sin(sb)\cos(sb)
\\
&=&
\ds\frac{1}{4}\left( \sinh(2sa)\sinh(2b) - \sin(2a)\sin(2sb) \right).
\end{eqnarray*}
Since $\sin(x) < x < \sinh(x)$ for any $x > 0$, 
if $ab>0$ we see that $\det(A) >0$ and the above system cannot have a non-trivial
solution. The same conclusion holds if $ab < 0$ since $\sin$ and $\sinh$
are odd functions.
We conclude that the dispersion relation~(\ref{disp_rel}) may have solutions only when
$\mu \in \R$ and when $\mu \in i\R$, {\it i.e.}
\[
\left\{\begin{array}{lcll}
\mu = a \in \R & \textrm{with}\quad & \tan(a) \;-\; \frac{s (N-D)}{D } \tanh(sa)\;=\; 0,
\\ [10pt]
\mu = ib/s \in i \R & \textrm{with}\quad & \tan(b) \;-\;  \ds\frac{D }{s (N-D)}\tanh(b/s)\;=\; 0.
\end{array}\right.
\]
Recalling that $s = \sqrt{|k^-|}$, there is exactly one
solution $a_n \in ]n\pi  , (n+1)\pi [$ to $$\tan(a_n) = \frac{s(N-D)}{D} \tanh(sa_n)$$
and exactly one solution $b_n  \in ]n\pi  , (n+1)\pi [$
to $$\tan(b_n) = \frac{D}{ s(N-D)}\tanh(b_n/s),$$ for $n \geq 1$.
It follows that we can construct two families of eigenfunctions.
The first corresponds to choosing $\mu_{n} = a_n, n \geq 1$ and

$$\phi^{2n}(x) = 
c_{2n}\big(\phi^{2n}_1,...,\phi_N^{2n}\big),$$
\begin{eqnarray*}
 \phi^{2n}_j=\left\{ \begin{array}{lr}
\ds\frac{\sin(a_n)}{\sinh(s a_n )}\sinh( s a_nx),
& \quad 1\leq j\leq D,
\\[8pt]
\sin(a_nx),
& \quad D+ 1\leq j\leq N,
\end{array}\right.
\end{eqnarray*}
where $c_{2n}$ is a normalization factor.
The second family corresponds to setting $\mu_{n} = ib_n/s$
or equivalently $\xi_{n} = b_n$ and the associated eigenfunctions are

$$\phi^{2n+1}(x) =
c_{2n+1}\big(\phi^{2n+1}_1,...,\phi_N^{2n+1}\big),$$
\begin{eqnarray*}
\phi^{2n+1}_j(x) =
\left\{ \begin{array}{lr}
\sin( b_n x),
& \quad 1\leq j\leq D  ,
\\[8pt]
\ds\frac{\sin(b_n)}{\sinh\big(\frac{b_n}{s}x\big)} \sinh\Bigg(\frac{b_n}{s}x\Bigg),
& \quad D+ 1\leq j\leq N,
\end{array}\right.
\end{eqnarray*}
with a normalization factor $c_{2n+1}$. 
\end{itemize}

Finally, the eigenvalues $(\lambda_n)_{n\in\N^*}$ of $A$ are obtained by reordering the sequences
$$(\nu_n)_{n\in\N^*},\ \ \ \ \ \ (\theta_n)_{n\in\N^*} ,\ \ \ \ \ \  \ (k^- a_n^2)_{n\in\N^*},  \ \ \ \ \  \ (b_n^2)_{n\in\N^*}.$$
We denote the corresponding eigenfunctions $(\psi^{n})_{n\in\N^*}$ obtained by reordering 
$$(\varphi^n)_{n\in\N^*},\ \ \ \ \ \ (\rho^n)_{n\in\N^*} ,\ \ \ \ \ \  \ (\phi^n)_{n\in\N^*}.$$

\begin{prop}\label{base}
The functions $(\psi^n)_{n \geq 1}$ form a complete orthonormal system of $L^2(\Si)$.
\end{prop}
\begin{proof} The result is a direct consequence of Theorem \ref{self_adjoint_A}.
\end{proof}

\begin{remark} \label{evolsyst1}
Let us consider the following Schr\"odinger type equation in $L^2(\Si,\C)$
\begin{equation}\label{system}
\begin{split}
\begin{cases}i\partial_t \varphi = A \varphi,\ \ \ \  \ \ \ \ \  \ \ t>0, \\
\varphi(t=0)=f\in L^2(\Si).\\
\end{cases}
\end{split}
\end{equation}
If we extend our analysis leading to Proposition \ref{base} for complex-valued function, then we can ensure the well-posedness of the problem \eqref{system} and express its solutions $\varphi(t)$ for any initial data $f\in L^2(\Si,\C)$ as follows:
$$S(t)f :=\varphi(t)=\sum_{j=1}^\infty e^{-i\lambda_jt }\psi^j ( \psi^j, f)_{L^2}.$$
We observe that $(S(t))_{t \in \mathbb{R}}$ is 
a unitary group of $L^2(\Si)$ generates by $-iA$ and, for all $f \in D(A)$, the system \eqref{system} admits a unique solution $$ \varphi \in C([0,+\infty);D(A)) \cap C^1([0,+\infty); L^2(\Si,\C)), \ \ \ \ \ \left\|\varphi(t)\right\|_{L^2} =\left\|f\right\|_{L^2}, \forall \, t \geq 0.$$
\end{remark}
\begin{remark}\label{evolsyst1_bis}
Now, consider the following heat-type equation
\begin{equation}\label{system_heat_half}
\begin{split}
\begin{cases}\partial_t u + \widetilde A u=0,\ \ \ \  \ \ \ \ \  \ t>0, \\
u(t=0)=f\in L^2(\Si,\R).\\
\end{cases}
\end{split}
\end{equation}
Due to the presence of negative eigenvalues, we can not define solutions of the dynamics \eqref{system_heat_half} in the whole space $L^2(\Si,\R)$. Nevertheless, we can introduce the space 
$$X=\{u\in L^2(\Si,\R):\ \exists N\in\N\ :\ \la u,\phi^{2j}\ra =\la u,\rho^{j}\ra =0,\ \ \forall j> N \}$$
which is dense in $L^2(\Si,\R)$. The dynamics of \eqref{system_heat_half} is well-defined in $X$ and the solutions $u(t)$, for any initial data $f\in X$, are decomposed as follows with suitable $N_f\in\N$:
\begin{align*}u(t)=&\sum_{j=1}^\infty e^{-\nu_jt }\varphi^j ( \varphi^j, f)_{L^2}+\sum_{j=1}^{N_f} e^{-\theta_jt }\rho^j ( \rho^j, f)_{L^2}\\
&+\sum_{j=1}^{N_f} e^{k^- a_j^2t }\phi^{2j} ( \phi^{2j}, f)_{L^2}+\sum_{j=1}^\infty e^{-b_j^2 t }\phi^{2j+1} (\phi^{2j+1}, f)_{L^2}.
\end{align*}
\end{remark}
\subsection{Standard spectral representation for some non-equilateral star-networks}\label{weird}
Let us consider the spectral problem \eqref{spec1} when the lengths $L_j=|e_j|$ of the edges of the star-graph are such that $L_j/L_k\not\in\Q$ for every distinct $j,k\leq N$ . As in the previous section, we study the eigenfunctions of the operator $ A$ and we set $\xi =  i s \mu$ with $s=\sqrt{|k^-|}$. The eigen-elements $\psi$ and $\xi$ 
equivalently satisfy
\begin{eqnarray*}
\psi  &=& \big(\alpha_1 \sin(\xi x),...,\alpha_{D}\sin(\xi x),
\beta_{D+1} \sin(\mu x), ...,\beta_{N} \sin(\mu x)\big),
\end{eqnarray*}
where the coefficients $ \alpha=(\alpha_1,...,\alpha_D)$ and $ \beta=(\beta_{D+1},...,\beta_N)$ satisfy
\begin{equation}\label{condis_bis}
\begin{split}
\begin{cases}\alpha_j \sin(\xi L_j) = \alpha_{j+1}\sin(\xi L_{j+1}),\ \ \ \  \ \ \ &1\leq j \leq D-1, \\
\alpha_D \sin(\xi L_D) = \beta_{D+1}\sin(\mu L_{D+1}), \\
\beta_{j} \sin(\mu L_j ) = \beta_{j+1}\sin(\mu L_{j+1} ),\ \ \ \  \ \ \ &D+ 1\leq j \leq N-1, \\
\ds \sum_{j=1}^D \alpha_j \xi\cos(\xi L_j) +k^- \ds \sum_{j=D+1}^N\beta_j \mu \cos(\mu L_j )=0.\\
\end{cases}
\end{split}
\end{equation}
Notice that, in contrast to the previous section, we have that $\sin(\xi L_j),\sin(\mu L_j)\neq 0$ for every $1\leq j\leq D$ and $D+ 1\leq l\leq N$. Indeed, if for instance $\sin(\xi L_j)= 0$ for $1\leq j\leq D$, then $\beta_l=0$ for all $D+ 1\leq l\leq N$ and there exists $1\leq m\leq D$ distinct from $j$ such that  $$\sin(\xi L_j)=\sin(\xi L_m)=0.$$ This identity implies that $\xi $ is of the form $$\frac{n_j\pi}{L_j} \ \ \ \ \ \text{and}\ \ \ \  \ \ \frac{n_m \pi }{L_m} \ \ \ \  \text{ with}\ \ \ \ n_j,n_m\in\Z^*,$$ which is impossible since $L_j/L_k\not\in\Q$ for every distinct $j,k\leq N$. This argument leads to  $$\sin(\xi L_j),\sin(\mu L_l)\neq 0,\ \ \ \  \ \ \ \  1\leq j\leq D,\ \ D+ 1\leq l\leq N,$$ and it also implies that the eigenfunctions of $A$ can not have any vanishing components thanks to the continuity condition in \eqref{condis_bis}. Now, the determinant corresponding to the linear system \eqref{condis_bis} yields

\begin{equation}\label{disp_relbis}\sum_{j=1}^D \xi \cot(\xi L_j) + k^-\sum_{j=D+1}^N \mu \cot(\mu L_j )=0.\end{equation}
We can observe that the eigenvalues are simple. Indeed, if $\lambda $ is a multiple eigenvalue, then there exist at least two distinct eigenfunctions $f$ and $g$ both corresponding to $\lambda $. Such eigenfunctions do not have any vanishing component and their first components can be respectively written as 
$$f_1=\alpha_1^f\sin (\sqrt\lambda  L_1),\ \ \ \ \ \ \  \ g_1=\alpha_1^g\sin (\sqrt\lambda  L_1).$$ Now, $h=\alpha^g_1 f - \alpha^f_1 g $ is another eigenfunction of $A$ corresponding to the eigenvalue $\lambda $ and its first component $h_1=\alpha^g_1\alpha_1^f\sin (\sqrt\lambda  L_1)-\alpha^f_1\alpha_1^g\sin (\sqrt\lambda  L_1) = 0$ which is absurd. 

\smallskip

Finally, the eigenvalues $(\lambda_n)_{n\in\N^*}$ of $A$ are simple, and they obtained by solving the transcendental equation \eqref{disp_relbis}. In other words, since $\mu =  -i \frac{\xi}{s}$ with $s=\sqrt{|k^-|}$, they are the solutions of 

$$\sum_{j=1}^D   \cot(\sqrt\lambda_n L_j) +   \sqrt{|k^-|}\sum_{j=D+1}^N   \coth\Big(\frac{\sqrt\lambda_n }{\sqrt{|k^-|}} L_j \Big)=0.$$
Some corresponding eigenfunctions $(\psi^n)_{n\in\N^*}$ have the form
\begin{equation*}
\begin{split}
\psi^n = \alpha_n \Bigg( & \sin(\sqrt\lambda_n x),  \frac{\sin(\sqrt\lambda_n L_1)}{\sin(\sqrt\lambda_n L_2)}\sin(\sqrt\lambda_n x), ...,  \frac{\sin(\sqrt\lambda_n L_1)}{\sin(\sqrt\lambda_n L_D)}\sin(\lambda_n x),\\
&  \frac{\sin(\sqrt\lambda_n L_1)}{\sinh\Big(\frac{\sqrt\lambda_n }{\sqrt{|k^-|} } L_{D+1} \Big)}\sinh\Big(\frac{\sqrt\lambda_n }{\sqrt{|k^-|} } x \Big),..., \frac{\sin(\sqrt\lambda_n L_1)}{\sinh\Big(\frac{\sqrt\lambda_n }{\sqrt{|k^-|} } L_{N} \Big)}\sinh\Big(\frac{\sqrt\lambda_n }{\sqrt{|k^-|} } x \Big)\Bigg) ,
\end{split}
\end{equation*}
with $\alpha_n$ some normalization constants. As in the previous section, the operator $A$ is self-adjoint with compact resolvent which implies the following result. 
 
\begin{prop}\label{base_bis}
The functions $(\psi^n)_{n \geq 1}$ form a complete orthonormal system of $L^2(\Si)$.
\end{prop}

\subsection{Spectral representation via pseudo-eigenfunctions for equilateral star-networks}
In the current section, we want to investigate the ``alternative'' spectral problem \eqref{spec2} when 
the edges of the star-graph have the same length. We assume for simplicity that $|e_j|=1$ for every $1\leq j\leq N$. Notice that an eigenfunction of $\widetilde A$ has the form
\begin{align}\label{eig}\psi=(\psi_1,...\psi_N),\ \ \  \ \ \ \  \ \psi_j(x) = \alpha_j \sin(\sqrt\lambda x),\ \ \ \ \ \ \ 1 \leq j\leq N,\end{align} where the coefficients $\alpha_j$ satisfy
$$\alpha_j \sin(\sqrt\lambda) = \alpha_N \sin(\sqrt\lambda), \quad\quad 1 \leq j \leq N,$$
$$\sum_{j=1}^D  \alpha_j \cos(\sqrt\lambda)+k^- \sum_{j=D+1}^N   \alpha_j \cos(\sqrt\lambda) = 0.$$
One easily checks that the determinant of this linear system is equal to
\begin{eqnarray*}
\sin(\sqrt\lambda)^{N-1}\cos(\sqrt\lambda)\big(-D -k^- (N-D)\big)
\end{eqnarray*}
and we recover the fact that the bilinear form $a_k$ is an isomorphism under the condition appearing in Theorem \ref{th_star_N} which is 
$$-k^- \neq \frac{D}{N-D}.$$
We find two families of eigenpairs. 
The first one $(\nu_n,\phi^n)_{n \in\N^*}$ is given by 
\begin{eqnarray} \label{bf_e}
\nu_n = \frac{(2n-1)^2\pi^2}{4}
&\textrm{and}&
\phi^{n}=(\phi^{n}_1,...,\phi^{n}_N),\ \ \ \ \ \  \ \phi^{n}_j=  \sin\Bigg(\frac{(2n-1)\pi}{2}x\Bigg).
\end{eqnarray}
Notice that $\cos(\sqrt\nu_n) = 0$ and 
the derivatives of $\phi^{n}$ vanish at the endpoint $x=1$ of each edge. The second family of eigenpairs corresponds to eigenvalues of multiplicity $N-1$:
$$\gamma_n = n^2 \pi^2,\ \ \ \  \ \ \ n\in\N^*.$$
The associated eigenspace to each $\gamma_n$ corresponds to eigenfunctions of the form
$$\varphi=( \alpha_{1} \sin(n\pi x),...,\alpha_{N} \sin(n\pi x)) \ \ \ \ \ \ \ \text{with}\ \ \ \ \ \ \ \ \ \sum_{j=1}^D \alpha_j +k^-\sum_{j=D+1}^N \alpha_j = 0.$$
Thus, we can construct the associated eigenfunctions as follows
$$\varphi=( k^- \widetilde \alpha_{1} \sin(n\pi x),..., k^- \widetilde \alpha_{D} \sin(n\pi x),\widetilde \alpha_{D+1} \sin(n\pi x),...,\widetilde \alpha_{N} \sin(n\pi x)), \ \ \ \ \text{with}\ \ \ \ \ \ \ \ \ \sum_{j=1}^N \widetilde \alpha_j = 0.$$
An example of basis for an eigenspace when $D= 1$ and $N=3$ is the one formed by the two orthogonal functions
$$( 0, - \sin(n\pi x), \sin(n\pi x) ),\ \ \ \ \ \ \ \ \ \ \ \ ( -2 k^- \sin(n\pi x) , \sin(n\pi x), \sin(n\pi x) ).$$

Finally, we call $(\lambda_n)_{n\in\N^*} $ the sequence of eigenvalues of $\widetilde A$ obtained by reordering the sequences $(\nu_n)_{n\in\N^*} $ and $(\gamma_n)_{n\in\N^*} $ and then
$$\lambda_n = \frac{n^2\pi^2}{4}.$$
Each eigenvalue $\lambda_n$ with $n\in 2\N^*$ has multiplicity $N-1$, while the others are simple.  We denote by $(\psi^n)_{n\in\N^*}$ a sequence of orthonormalized eigenfunctions corresponding to $(\lambda_n)_{n\in\N^*} $.

\begin{prop}\label{prop-Riesz}
The family of functions $(\psi^n)_{n\in\N^*}$ forms a Riesz Basis of $L^2(\Si)$ , {\it i.e.} it is isomorphic to an orthonormal system of $L^2(\Si)$.
\end{prop}
\begin{proof}
First, notice that the two families of eigenfunctions respectively corresponding to the eigenvalues $(\nu_n)_{n\in\N^*} $ and $(\gamma_n)_{n\in\N^*} $ are mutually orthogonal  since $$\Big(\sin(n\pi x),\sin\Big(\frac{2l -1}{2}\pi x\Big)\Big)_{L^2}=0,\ \ \ \ \ \ \ \ \ \  \ \ \forall n,l\in\N^*.$$ We denote by $\Hi_1$ and $\Hi_2$ the corresponding eigenspaces. Second, we introduce the linear mapping $T:\Hi_1 + \Hi_2\rightarrow L^2(\Si)$ such that $T|_{\Hi_1}=Id_{\Hi_1}$ and $T|_{\Hi_2}$ is defined as follows
$$T\psi=(\frac{1}{ k^- }\psi_1,...,\frac{1}{ k^- }\psi_D, \psi_{D+1},..., \psi_N),\ \ \ \ \ \ \  \ \psi\in \Hi_2.$$
We notice that $T$ is an isomorphism from $\Hi_1+\Hi_2$ to $T(\Hi_1+\Hi_2)$.  Now,  $(\phi^n)_{n\in\N^*}$ with $\phi^j=S(T\psi^j)$ are eigenfunctions of a standard Dirichlet Laplacian defined on a star-graph with edges of equal length (see for instance \cite[Section 4.1]{mioecri}) which is self-adjoint with compact resolvent. Finally, $(\phi^n)_{n\in\N^*}$ define a complete orthonormal system, $\Hi_1+\Hi_2=L^2(\Si,\R)$ and $(\psi^n)_{n\in\N^*}$ forms a Riesz Basis in $L^2(\Si)$.
\end{proof}

\begin{remark} \label{evolsyst2}
Similarly to Remark \ref{evolsyst1}, we can exploit Proposition \ref{prop-Riesz} to characterize some evolution equations defined with the operator $\widetilde A$. Consider for instance the following Schr\"odinger-type equation in $L^2(\Si,\C)$
\begin{equation}\label{system_sch}
\begin{split}
\begin{cases}i\partial_t v = \widetilde A v,\ \ \ \  \ \ \ \ \  \ \ t>0, \\
v(t=0)=f\in L^2(\Si).\\
\end{cases}
\end{split}
\end{equation}
Notice that our analysis leading to Proposition \ref{prop-Riesz} is still valid for complex valued functions.
First, we introduce the biorthogonal family in $L^2(\Si,\C)$ of the Riesz basis $(\psi^n)_{n\in\N^*}$ : this is the unique sequence of functions $(\sigma^n)_{n\in\N^*}$ such that
$$\big(\psi^n,\sigma^k\big)_{L^2}=\delta_{k,n},  \ \ \  \ \ \ \  \ \ \  \ \forall k\in\N^*.$$
The functions $\sigma^n$ can be explicitly computed by using the orthogonal projectors\, $\pi_k:L^2\longrightarrow \overline{\spn\{\psi^j:j\neq k\}}^{\, L^2}$ as follows:
$$\sigma^k=\frac{\psi^k-\pi_k \psi^k}{\|\psi^k-\pi_k\psi^k\|^{2}_{L^2}},\ \ \ \ \ \ \ \forall k\in\N^*.$$ 
Second, we ensure the well-posedness of the evolution equation \eqref{system_sch} by using the properties of the Riesz basis and its solutions $v(t)$ for any initial data $f\in L^2(\Si,\C)$ are of the form:
$$T(t)f :=v(t)=\sum_{j=1}^\infty e^{-i\lambda_jt }\psi^j ( \sigma^j, f)_{L^2}.$$
We have that $(T(t))_{t \in \mathbb{R}}$ is 
a group of $L^2(\Si,\C)$ generates by $-i\tilde{A}$ and, for all $f \in D(\tilde{A})$, the system \ref{system_sch} admits a unique solution $$ v \in C([0,+\infty);D(\tilde{A})) \cap C^1([0,+\infty); L^2(\Si,C)),\ \ \ \ \ \ \left\|v(t)\right\|_{L^2} =\left\|f\right\|_{L^2}, \forall \, t \geq 0.$$

\end{remark}

\begin{remark} \label{evolsyst2_bis}
Consider now the following heat-type equation in $L^2(\Si, \R)$
\begin{equation}\label{system_heat}
\begin{split}
\begin{cases}\partial_t u + \widetilde A u=0,\ \ \ \  \ \ \ \ \  \ t>0, \\
u(t=0)=f\in L^2(\Si,\R).\\
\end{cases}
\end{split}
\end{equation}
Equivalently to the complex case (Remark \ref{evolsyst2}), we can define a family of eigenfunctions $(\psi^n)_{n\in\N^*}$ of $\widetilde A$ forming a Riesz basis of $L^2(\Si, \R)$ and we denote by $(\sigma^n)_{n\in\N^*}$ its biorthogonal family. Now, the solutions $u(t)$ of the problem \eqref{system_heat} for any initial data $f\in L^2(\Si,\R)$ are:
$$R(t)f :=u(t)=\sum_{j=1}^\infty e^{-\lambda_jt }\psi^j ( \sigma^j, f)_{L^2}.$$
This solution is a semigroup solution and, when $f \in D(\tilde{A}),$ 
$$u \in C([0,+\infty);D(\tilde{A})) \cap C^1([0,+\infty); L^2(\Si)).$$
Notice that, on the contrary for the operator $A$ considered in Remark \ref{evolsyst1_bis}, the eigenvalues of $\widetilde A$ are positive, which ensures the fact the solutions are well-defined.  
\end{remark}

\subsection{Spectral representation via pseudo-eigenfunctions for some non-equilateral star-networks}
This section aims to study \eqref{spec2} when the lengths $L_j=|e_j|$ of the edges of the star-graph are such that $L_j/L_k\not\in\Q$ for every distinct $j,k\leq N$ . As in the previous section, we study the eigenfunctions of the operator $\widetilde A$ which have the form
\begin{align}\label{eig_bis}\psi=(\psi_1,...\psi_N),\ \ \  \ \ \ \  \ \psi_j(x) = \alpha_j \sin(\sqrt\lambda x),\ \ \ \ \ \ \ 1 \leq j\leq N,\end{align} where the coefficients $\alpha_j$ satisfy
$$\alpha_j \sin(\sqrt\lambda L_j) = \alpha_N \sin(\sqrt\lambda L_N), \quad\quad 1 \leq j \leq N,$$
$$\sum_{j=1}^D  \alpha_j \cos(\sqrt\lambda L_j)+k^- \sum_{j=D+1}^N   \alpha_j \cos(\sqrt\lambda L_j ) = 0.$$
On the contrary of the previous section, we can observe that the coefficients $$\alpha_j\neq 0,\ \ \ \  \ \ \forall 1\leq j\leq N.$$ Indeed, if one of these coefficients vanishes, then there exist distinct $l,m\leq N$ such that $\alpha_l,\alpha_m\neq 0$, and $$\sin(\sqrt\lambda L_l)=\sin(\sqrt\lambda L_m)=0.$$ This identity implies that $\lambda$ is of the form $$\frac{n_l^2\pi^2}{L_l^2} \ \ \ \ \ \text{and}\ \ \ \  \ \ \frac{n_m^2\pi^2}{L_m^2} \ \ \ \  \text{ with}\ \ \ \ n_l,n_m\in\Z^*,$$ which is impossible since $L_j/L_k\not\in\Q$ for every distinct $j,k\leq N$. This argument not only yields that all the components of $\psi$ are different from $0$, but also that $$ \sin(\sqrt\lambda L_l)\neq 0,\ \ \  \ \  \ \ \ \ \ \forall 1\leq l\leq N,$$ and it can also be used to prove that
$$ \cos(\sqrt\lambda L_l)\neq 0,\ \ \  \ \  \ \ \ \ \ \forall 1\leq l\leq N.$$
In addition, we have that each eigenvalue $\lambda$ is simple since the existence of multiple eigenvalues would imply the existence of eigenfunctions vanishing in at least one edge which is impossible (as in Section \ref{weird}).

\smallskip

Finally, we can define the eigenpairs $(\lambda_n,\psi^n)_{n \in\N^*}$ where $(\lambda_n)_{n\in\N^*}$ is a sequence of simple eigenvalues defined by the relation
$$\sum_{j=1}^D  \cot(\sqrt\lambda_n L_j)+k^- \sum_{j=D+1}^N    \cot(\sqrt\lambda_n L_j ) = 0,$$
while $(\psi^n)_{n\in\N^*}$ is a sequence of eigenfunctions such that
$$\psi^n=\alpha_n\Big( \frac{\sin(\sqrt\lambda_n L_N)}{\sin(\sqrt\lambda_n L_1)}\sin(\lambda_n x), ... ,\frac{\sin(\sqrt\lambda_n L_{N})}{\sin(\sqrt\lambda_n L_{N-1})}\sin(\sqrt\lambda_n x), \sin(\sqrt\lambda_n x)\Big),$$
with $\alpha_n$ some normalizing constants.
\begin{prop}\label{prop-Riesz_1}
The family of functions $(\psi^n)_{n\in\N^*}$ forms a Riesz Basis of $L^2(\Si)$, {\it i.e.} it is isomorphic to an orthonormal system of $L^2(\Si)$.
\end{prop}
\begin{proof}
The result is ensured equivalently to Proposition \ref{prop-Riesz}. Indeed, the functions $\phi_j=T\psi_j$ are the eigenfunctions of the standard Dirichlet Laplacian defined on star graphs with edges of different lengths which is self-adjoint with compact resolvent (see again \cite[Section 4.1]{mioecri}).
\end{proof}

\section{Spectral representation for star-networks with mixed boundaries}\label{spec_star_neumann}
This section aims to retrace the theory of Section \ref{spec_star} when some external vertices of the star-graph are equipped with Neumann boundary conditions. In detail, we consider the two spectral problems \eqref{spec1} and \eqref{spec2} and we study them with functions in $\widehat H^1(\Si)$ when the boundary conditions in \eqref{mainxSNK_Neumann} are verified.  To simplify the theory, we assume $N_d^+=N_n^+=N_d^-=N_n^-=2$, but the same approach can be used in the general case.

\subsection{Standard spectral representation for equilateral star-networks}\label{spec_diff_neumann}
Let us consider the spectral problem \eqref{spec1} when the length of all the edges is equal to $L=1$. The same results can be directly extended to the case where all the edges have the same length. 
\begin{itemize}

\item We find the eigenvalues $\lambda$ of the form $$n^2\pi^2,\ \ \ \ \ \ \ \forall n\in\N^*$$ and corresponding to the eigenfunctions of the form
$$\psi= \big(\sin(n \pi x),-\sin(n \pi x),0,0,0,0,0,0\big).$$
Notice that such eigenvalues are simple due to the choice of $N_d^+=2$. In the general case, the same type of construction can be repeated and the multiplicity is $N_d^+-1$.
\item We find the eigenvalues $\lambda$ of the form $$\frac{(2n-1)^2\pi^2}{4},\ \ \ \ \ \ \ \forall n\in\N^*$$ and corresponding to the eigenfunctions of the form
$$\psi= \Big(0,0,\cos\big(\frac{(2n-1) \pi}{2} x\big),-\cos\big(\frac{(2n-1) \pi}{2} x\big),0,0,0,0,0\Big).$$
These eigenvalues are simple and, in the general case, their multiplicity is $N_d^+-1$.

\item We find the eigenvalues $\lambda$ of the form $$k^-n^2\pi^2,\ \ \ \ \ \ \ \forall n\in\N^*$$  and corresponding to the eigenfunctions of the form
$$\psi=\big(0,0,0,0, \sin(n \pi x),- \sin(n \pi x),0,0\big),$$
As above, the eigenvalues are simple (in the general case the multiplicity is $N_d^--1$).
\item We find the eigenvalues $\lambda$ of the form $$k^-\frac{(2n-1)^2\pi^2}{4},\ \ \ \ \ \ \ \forall n\in\N^*$$  and corresponding to the eigenfunctions of the form
$$\psi=\Big(0,0,0,0,0,0, \cos\big(\frac{(2n-1) \pi}{2} x\big),- \cos\big(\frac{(2n-1) \pi}{2} x\big)\Big),$$
The eigenvalues are simple (the multiplicity is $N_d^--1$).

\item The last family of eigenvalues $\lambda\in\R$ are not of the form
$$\frac{n^2\pi^2}{4}, \ \ \ k^-\frac{n^2\pi^2}{4},\ \ \ \ \ \ \ \forall n\in\N^*.$$
They are defined as the simple zeros of the equation\begin{equation}\begin{split}\label{nini}&N_d^+\cot(\sqrt{\lambda})-N_n^+ \tan(\sqrt{\lambda})-\sqrt{|k^-|}  N_d^-\coth\Big(\frac{\sqrt{\lambda}}{\sqrt{|k^-|}} \Big)-\sqrt{|k^-|} N_n^-\tanh\Big(\frac{\sqrt{\lambda}}{\sqrt{|k^-|}} \Big)\\
 &=2\cot(\sqrt{\lambda})- 2\tan(\sqrt{\lambda})-2\sqrt{|k^-|} \coth\Big(\frac{\sqrt{\lambda}}{\sqrt{|k^-|}} \Big)-2\sqrt{|k^-|}\tanh\Big(\frac{\sqrt{\lambda}}{\sqrt{|k^-|}} \Big)=0.\end{split}\end{equation}The corresponding eigenfunctions have the form
$$\psi (x) =
 \big(\psi_1,...,\psi_8 \big),$$
\begin{eqnarray*}
\psi_j(x) =
\left\{ \begin{array}{lr}
\sin( \sqrt{\lambda}  x)
& \quad 1\leq j\leq 2  ,
\\[8pt]
\ds\frac{\sin(\sqrt{\lambda} )}{\cos(\sqrt{\lambda})} \cos(\sqrt{\lambda} x)
& \quad 3\leq j\leq 4,\\[8pt]
\ds\frac{\sin(\sqrt{\lambda})}{\sinh\big(\frac{\sqrt{\lambda}}{\sqrt{|k^-|}}x\big)} \sinh\Bigg(\frac{\sqrt{\lambda}}{\sqrt{|k^-|}}x\Bigg)
& \quad 5\leq j\leq 6,\\[8pt]
\ds\frac{\sin(\sqrt{\lambda} )}{\cosh\big(\frac{\sqrt{\lambda} }{\sqrt{|k^-|}}x\big)} \cosh\Bigg(\frac{\sqrt{\lambda} }{\sqrt{|k^-|}}x\Bigg)
& \quad 7\leq j\leq 8.\\
\end{array}\right.
\end{eqnarray*}
The definition of this eigenpairs can be directly extended to the general case of $\N_d^+,\N_n^+,\N_d^-,\N_n^-\in\N^*.$

\end{itemize}

\subsection{Standard spectral representation for some non-equilateral star-networks}\label{weird_neumann}
We consider the spectral problem \eqref{spec1} in the general case when $L_j/L_k\not\in\Q$ and we proceed as in Section \ref{weird}. The eigenvalues $\lambda\in\R$ in such a framework are simple, and they are defined as the zeros of the identity
\begin{equation}\label{eqeq}\begin{split}&0=\sum_{j=1}^{2}  \cot(\sqrt{\lambda} L_j)-\sum_{j=3}^{4} \tan(\sqrt{\lambda} L_j)-\\
& \sqrt{|k^-|}\sum_{j=5}^{6}  \coth\Big(\frac{\sqrt{\lambda}}{\sqrt{|k^-|}} L_j \Big)-\sqrt{|k^-|}\sum_{j=7}^8 \tanh\Big(\frac{\sqrt{\lambda}}{\sqrt{|k^-|}} L_j \Big).\end{split}\end{equation}
The corresponding eigenfunctions have the form

$$\psi (x) =
 \big(\psi_1,...,\psi_8 \big),$$
\begin{eqnarray*}
\psi_j(x) =
\left\{ \begin{array}{lr}
\sin( \sqrt{\lambda}  x)
& \quad 1\leq j\leq 2  ,
\\[8pt]
\ds\frac{\sin(\sqrt{\lambda} )}{\cos(\sqrt{\lambda})} \cos(\sqrt{\lambda} x)
& \quad 3\leq j\leq 4,\\[8pt]
\ds\frac{\sin(\sqrt{\lambda})}{\sinh\big(\frac{\sqrt{\lambda}}{\sqrt{|k^-|}}x\big)} \sinh\Bigg(\frac{\sqrt{\lambda}}{\sqrt{|k^-|}}x\Bigg)
& \quad 5\leq j\leq 6,\\[8pt]
\ds\frac{\sin(\sqrt{\lambda} )}{\cosh\big(\frac{\sqrt{\lambda} }{\sqrt{|k^-|}}x\big)} \cosh\Bigg(\frac{\sqrt{\lambda} }{\sqrt{|k^-|}}x\Bigg)
& \quad 7\leq j\leq 8.\\
\end{array}\right.
\end{eqnarray*}

\begin{remark}
Notice that the equation \eqref{eqeq} can be extended to the general case of  $N_d^+,N_n^+,N_d^-,N_n^-\in\N^*$ by changing the extremes of summation of each term. The construction of the eigenfunctions follows equivalently. \end{remark}

\subsection{Spectral representation via pseudo-eigenfunctions for equilateral star-networks}
In the current section, we study the spectral problem \eqref{spec2} when the edges of the star-graph have the same length $L=1$ and in the presence of some Neumann boundary conditions. 
\begin{itemize}
\item We can define two families of eigenvalues. The first are the numbers $\lambda$ of the form 
$$\frac{(2n-1)^2\pi^2}{4},\ \ \ \ \ \ \forall n\in\N^*,$$
and with multiplicity $4$.  The corresponding eigenfunctions $\psi$ have the following forms
$$\Big( \sin(\sqrt{\lambda}x),  \sin(\sqrt{\lambda}x),1,1,\sin(\sqrt{\lambda}x),\sin(\sqrt{\lambda}x),1,1\Big),$$
$$\Big(0, 0,\cos(\sqrt{\lambda}x) ,-\cos(\sqrt{\lambda}x),0,0,0,0\Big),$$
$$\Big(0, 0,0,0,0,0,\cos(\sqrt{\lambda}x) ,-\cos(\sqrt{\lambda}x)\Big),$$
$$\Big(0, 0,-k^-\cos(\sqrt{\lambda}x) ,-k^-\cos(\sqrt{\lambda}x),0,0,\cos(\sqrt{\lambda}x),\cos(\sqrt{\lambda}x)\Big).$$
For general $N_d^+,N_n^+,N_d^-,N_d^-\in\N^*$, the $4$ type of eigenfunctions can be defined equivalently: the first type is always simple, the second has multiplicity $N_n^+-1$, the third has multiplicity $N_n^--1$, and the fourth is always simple.

\smallskip

\item  The second family of eigenvalues is composed of numbers $\lambda$ of the form:
$$ n^2 \pi^2,\ \ \ \  \ \ \ n\in\N^*$$
with multiplicity $3$ (in the general case $N_d^++N_d^--1$). The corresponding eigenfunctions have  the form
$$\psi=(\psi^1,\psi^2,0,0,\psi^5,\psi^6,0,0),\ \ \  \ \ \ \ \psi^j(x)  =  \alpha_{j} \sin(n\pi x)$$
$$\text{with}\ \ \ \ \ \ \ \ \ \alpha_1+\alpha_2 + k^-\alpha_5+k^-\alpha_6 = 0.$$
An example is given by the following three orthogonal functions
$$( - \sin(n\pi x), \sin(n\pi x),0,0,0,0,0,0 ),\ \ \ \ \ \ \ \ \ \ \ \ (0,0,0,0, - \sin(n\pi x), \sin(n\pi x),0,0 ),$$
$$(-k^-\sin(n\pi x), -k^-\sin(n\pi x),0,0, \sin(n\pi x),\sin(n\pi x),0,0 ).$$
\end{itemize}

\subsection{Spectral representation via pseudo-eigenfunctions for some non-equilateral star-networks}
Finally, we study \eqref{spec2}  when $L_j=|e_j|$ are such that $L_j/L_k\not\in\Q$ for every distinct $j,k\leq N$ . In this case, the eigenvalues $\lambda$ are simple (this also true for general $N_d^+,N_n^+,N_d^-,N_n^-$) and they are the zeros of the equation
$$\sum_{j=1}^2  \cot(\sqrt\lambda L_j)-\sum_{j=3}^4  \tan(\sqrt\lambda L_j)+k^- \sum_{j=5}^6    \cot(\sqrt\lambda L_j ) -k^- \sum_{j=7}^8    \tan(\sqrt\lambda L_j ) = 0.$$
The corresponding eigenfunctions have the form

$$\psi (x) =
 \big(\psi_1,...,\psi_8 \big),$$
\begin{eqnarray*}
\psi_j(x) =
\left\{ \begin{array}{lr}
\sin( \sqrt{\lambda}  x)
& \quad 1\leq j\leq 2  ,
\\[8pt]
\ds\frac{\sin(\sqrt{\lambda} )}{\cos(\sqrt{\lambda})} \cos(\sqrt{\lambda} x)
& \quad 3\leq j\leq 4,\\[8pt]
\ds\frac{\sin(\sqrt{\lambda})}{\sinh\big(\frac{\sqrt{\lambda}}{\sqrt{|k^-|}} \big)} \sin\Bigg(\frac{\sqrt{\lambda}}{\sqrt{|k^-|}}x\Bigg)
& \quad 5\leq j\leq 6,\\[8pt]
\ds\frac{\sin(\sqrt{\lambda} )}{\cosh\big(\frac{\sqrt{\lambda} }{\sqrt{|k^-|}}\big)} \cos\Bigg(\frac{\sqrt{\lambda} }{\sqrt{|k^-|}} x\Bigg)
& \quad 7\leq j\leq 8.\\
\end{array}\right.
\end{eqnarray*}

\section{Spectral representation for two-phase tadpole networks}\label{spec_tad}

This section aims to study the spectral representations considered in Section \ref{spec_star} for the case of the tadpole graphs $\Ti$ introduced in Section \ref{simple}.  The case in which the external vertex is not equipped with Dirichlet boundary conditions but with the Neumann one can be treated in the same way.

\smallskip

Assume that the lengths of the edges are $|e_1|=L_1$ and $|e_2|=L_2$. We choose the constant conductivities: $k_1=1$ and $k_2=k^-<0$.  As in the previous section, we introduced the quadratic form
\begin{eqnarray*}
(\psi,\vf)_k &=& \int_{e_1} \psi_1 \vf_1
\;+\; \int_{e_2}  k^-\psi_2 \vf_2.
\end{eqnarray*}
and, again, one can specify two different interesting spectral problems.

\begin{itemize}
    \item We aim to find $\lambda$ and $\psi \in H^1_0({\Ti})$
such that, for all $\vf \in  H^1_0({\Ti})$,
\begin{eqnarray}\label{spec1_tad}
a_k(\psi,\vf) &:=&    \int_{e_1}  \psi_1^\prime \vf_1^\prime
\;+\;
 \int_{e_2}  k^-\psi_2^\prime \vf_2^\prime
\;=\;  \lambda  (\psi,\vf).
\end{eqnarray}
It corresponds to the study the eigenfunctions of the following operator in $L^2(\Ti)$:
$$A\psi =\big(-\partial_x^2\psi_1, \;-k^-\partial_x^2\psi_{2} \big) \ \ \  \ \ \text{for}\ \ \  \psi\in D(A) \ \ \ \ \ \ \text{with}$$
\begin{align*}
D(A)=\Big\{\psi\in  H^{2}(e_1,\R)\times H^{2}(e_2,\R)\ :\ &\psi_2(0)=0,\ \ \psi_1(0)=\psi_1(|e_1|)= \psi_N(|e_2|),\\
&  \partial_x\psi_1(0)-\partial_x\psi_1(|e_1|)-k^-  \partial_x\psi_2(|e_2|) =0\Big\}.
\end{align*}

\begin{prop}\label{self_adjoint_A_tadpole}
The operator $A$ is self-adjoint and admits compact resolvent for every $k^-<0$.
\end{prop}
\begin{proof}
First, the fact that $A$ is self-adjoint follows exactly as in the proof of Proposition  \ref{self_adjoint_A}. Second, $A$ admits compact resolvent thanks to Theorem \ref{th_tadpole}.
\end{proof}

\item The second ``alternative''  eigenvalue problem is: find $\lambda$ and $\psi \in H^1_0({\Ti})$
such that, for all $\vf \in  H^1_0({\Ti})$,
\begin{eqnarray}\label{spec2_tad}
a_k(\psi,\vf) \;=\;  \lambda  (\psi,\vf)_k.
\end{eqnarray}
It corresponds to study the spectral problem associated with the operator:
$$\widetilde A\psi =\big(-\partial_x^2\psi_1\;, \;-\partial_x^2\psi_{2} \big) \ \ \  \ \ \text{for}\ \ \  \psi\in D(\widetilde A) \ \ \ \ \ \ \text{with}$$
\begin{align*}
D(\widetilde A)=\Big\{\psi\in  H^{2}(e_1,\R)\times H^{2}(e_2,\R)\ :\ &\psi_2(0)=0,\ \ \psi_1(0)=\psi_1(|e_1|)= \psi_N(|e_2|),\\
&  \partial_x\psi_1(0)-\partial_x\psi_1(|e_1|)-k^-  \partial_x\psi_2(|e_2|) =0\Big\}.
\end{align*}

\begin{remark} As in Remark \ref{non_self_adjoint_A}, the operator $\widetilde A$ is neither self-adjoint, nor symmetric. \end{remark}
 \end{itemize}

 \subsection{Standard spectral representation}

Let us start by considering the spectral problem \eqref{spec1_tad}.  We exploit the ideas developed in the proof of Theorem \ref{th_tadpole} and the eigenfunctions are either symmetric or antisymmetric w.r.t. symmetry axe $r$. This fact allows us to distinguish between two families of eigenpairs. The first $(\nu_j,\varphi^j)_{j\in\N^*}$ is composed of antisymmetric functions:
$$\nu_j=\frac{4j^2\pi^2}{L_1^2},\ \ \ \ \ \ \  \ \ \ \ \  \ \ \varphi^{j}=\Big(\alpha_j \sin\Big(\frac{2j \pi }{L_1 }x\Big),0\Big).$$
with some normalizing $\alpha_j$. The second $(\theta_j,\phi^j)_{j\in\N^*}$ is instead associated to symmetric eigefunctions. Each eigenvalue, in this case, is associated to a function of form:
$$\Big(\beta_1\cos\Big(\xi \Big(x-\frac{L_1}{2}\Big)\Big), \beta_2\sin (\mu x ) \Big),$$
with $\xi =  i s \mu$ and $s=\sqrt{|k^-|}$. Now, the boundary conditions defining $D(A)$ yield
\begin{equation}\label{condis_tad1}
\begin{split}
\begin{cases} 
\beta_{1} \cos\Big(\xi \frac{L_1}{2} \Big) =  \beta_{2}\sin (\mu L_{2} ), \\
2 \beta_1\xi \sin\Big(\xi \frac{L_1}{2}\Big) -k^-  \beta_{2} \mu \cos (\mu L_{2} )=0.\\
\end{cases}
\end{split}
\end{equation}
These identities tell us that $\beta_1,\beta_2\neq 0$ and $\cos(\xi \frac{L_1}{2}),\sin (\mu L_{2}  )\neq 0 $. Thus, by recalling $\mu = - i\frac{ \xi}{\sqrt{|k^-|}}$, the eigenvalues $\theta_j$ are those numbers solving the following equation
$$2 \tan\Big(\sqrt\theta_j\frac{L_1}{2}\Big) + \sqrt{|k^-|}    \coth\Big(\frac{\sqrt\theta_j}{\sqrt{|k^-|}} L_{2} \Big)=0.$$
Notice that $(\theta_j)_{j\in\N^*}$ are also real thanks to the self-adjointness of $A$ and simple thanks to the same arguments adopted in Section \ref{weird}. The eigenfunctions $\phi^j$ can be defined via some normalizing constants $\alpha_j$ as follows:
$$\phi^j=\alpha_j\Bigg(\cos\Big(\sqrt\theta_j \Big(x-\frac{L_1}{2}\Big)\Big),  \frac{\cos\Big(\sqrt\theta_j \frac{L_1}{2} \Big)}{\sinh \Big(\frac{\sqrt\theta_j}{\sqrt{|k^-|}} L_{2} \Big)}\sinh\Big(\frac{\sqrt\theta_j}{\sqrt{|k^-|}} x\Big)  \Bigg).$$
Finally, the eigenvalues $( \lambda_j)_{j\in\N^*}$ of $A$ are obtained by reordering $(\nu_j)_{j\in\N^*}$ and $(\theta_j)_{j\in\N^*}$. We denote by $(\psi^j)_{j\in\N^*}$ the corresponding family of eigenfunctions obtained by reordering $(\varphi^j)_{j\in\N^*}$ and $(\phi^j)_{j\in\N^*}$.

\begin{prop}\label{base_tad}
The functions $(\psi^n)_{n \geq 1}$ form a complete orthonormal system of $L^2(\Ti)$.
\end{prop}

\subsection{Pseudo-eigenfunctions for a tadpole network}
As above, we distinguish between two families of eigenpairs. The first $(\nu_j,\varphi^j)_{j\in\N^*}$ is the same defined for the operator $A$. The second eigenpair instead $(\widetilde \theta_j,\widetilde \phi^j)_{j\in\N^*}$ is associated to eigefunctions of form
$$\Bigg(\widetilde\beta_1\cos\Big( \sqrt{\widetilde\theta} \Big(x-\frac{L_1}{2}\Big)\Big), \widetilde\beta_2\sin (  \sqrt{\widetilde\theta}  x ) \Bigg).$$
such that
\begin{equation}\label{condis_tad1bis}
\begin{split}
\begin{cases} 
\widetilde \beta_{1} \cos\Big( \sqrt{\widetilde\theta} \frac{L_1}{2} \Big) =  \widetilde \beta_{2}\sin  (  \sqrt{\widetilde\theta}  L_{2}  ),\ \ \ \  \  \\
2 \widetilde \beta_{1} \sin\Big( \sqrt{\widetilde\theta} \frac{L_1}{2}\Big) -k^-  \widetilde \beta_{2} \cos (  \sqrt{\widetilde\theta}  L_{2}  )=0\\ 
\end{cases}
\end{split}
\end{equation}
with $ \widetilde \beta_1, \widetilde \beta_2\neq 0$ and. Notice that, if
\begin{equation}\label{id1} \cos\Big( \sqrt{\widetilde\theta}\frac{L_1}{2} \Big) =\sin (  \sqrt{\widetilde\theta} L_{2} )= 0, \end{equation}
then $ {\widetilde\theta} $ has the form 
$$ \frac{(2n_1-1)^2\pi^2}{L_1^2}\ \ \ \ \ \ \ \ \text{and}\ \ \ \  \   \ \  \frac{n_2^2\pi^2}{L_2^2},$$
which is possible only  when $$\frac{L_1}{L_2}\in \Q.$$
Some corresponding eigenfunctions have the form, with $\widetilde\alpha $ suitable, 
$$\widetilde\phi=\Bigg(\cos\Big( \sqrt{\widetilde\theta} \Big(x-\frac{L_1}{2}\Big)\Big),\widetilde\alpha \sin \Big( \sqrt{\widetilde\theta}  x \Big)\Bigg).$$
When $ \cos\big( \sqrt{\widetilde\theta}\frac{L_1}{2} \big)\neq 0$ and $\sin (  \sqrt{\widetilde\theta} L_{2} )\neq  0, $ we can define the eigenvalues by the relation 
\begin{equation}\label{id2}2   \tan\Big( \sqrt{\widetilde\theta} \frac{L_1}{2}\Big) -k^-   \cot (  \sqrt{\widetilde\theta}  L_{2}  )=0\end{equation}
In this case, some corresponding eigenfunctions are defined by
$$\widetilde\phi=\Bigg(\cos\Big( \sqrt{\widetilde\theta} \Big(x-\frac{L_1}{2}\Big)\Big), \frac{\cos\Big( \sqrt{\widetilde\theta} \frac{L_1}{2} \Big) }{\sin  (  \sqrt{\widetilde\theta}  L_{2}  )}\sin \Big( \sqrt{\widetilde\theta}  x \Big)\Bigg).$$
Now, the sequence of eigenvalues $(\widetilde \theta_j)_{j\in\N^*}$ is defined by reordering the solutions of \eqref{id1} and of \eqref{id2} and we denote by $(\widetilde \phi^j)_{j\in\N^*}$ some corresponding normalized eigenfunctions.

\smallskip

Finally, the eigenvalues $(\widetilde \lambda_j)_{j\in\N^*}$ of $\widetilde A$  are obtained by reordering $(\nu_j)_{j\in\N^*}$ and $(\widetilde \theta_j)_{j\in\N^*}$. We denote by $(\widetilde \psi^j)_{j\in\N^*}$ the corresponding family of eigenfunctions obtained by reordering $(\varphi^j)_{j\in\N^*}$ and $(\widetilde \phi^j)_{j\in\N^*}$.

\smallskip

As done for the star graphs with Proposition \ref{prop-Riesz} and Propositions \ref{prop-Riesz_1}, we can see now that $(\widetilde \psi^j)_{j\in\N^*}$ are isomorphic to the eigenfunctions of the Dirichlet Laplacian defined on a tadpole graph (see \cite[Section 4.2]{mioecri}), which yields the following proposition
\begin{prop}\label{prop-Riesz_1bis}
The family of functions $(\widetilde \psi^n)_{n\in\N^*}$ forms a Riesz Basis of $L^2(\Ti)$, {\it i.e.} it is isomorphic to an orthonormal system of $L^2(\Ti)$.
\end{prop}

\end{document}